\documentclass[a4paper,11pt]{article}
\usepackage{amssymb}
\usepackage{amsmath}
\usepackage{amsthm}
\usepackage{listings}
\usepackage{amsthm}
\usepackage{bbm}
\numberwithin{equation}{section}
\usepackage{slashed}
\usepackage[font=footnotesize]{caption}
\usepackage[english,german]{babel}
\usepackage[utf8]{inputenc}
\usepackage{fancybox}
\usepackage{a4wide}
\usepackage{graphicx}
\usepackage{fancyhdr} 
\usepackage{hyperref}
\usepackage{mathrsfs}
\usepackage{subfigure}
\usepackage{extarrows}
\usepackage{lipsum}
\usepackage{ulem}
\usepackage{xcolor}
\hypersetup{colorlinks=true,
	linkcolor={blue!90!black},
    	citecolor={blue!50!black},
    	urlcolor={blue!80!black}
   	 }

\newtheorem{theorem}{Theorem}[section]
\newtheorem{lemma}[theorem]{Lemma}
\newtheorem{proposition}[theorem]{Proposition}
\newtheorem{corollary}[theorem]{Corollary}

\theoremstyle{definition}

\newtheorem{assumption}[theorem]{Remark}

\theoremstyle{remark}

\newcommand{\pM}{\mathbb{P}}

\definecolor{darkred}{rgb}{0.5,0.0,0.1} 
\definecolor{darkgreen}{rgb}{0,0.4,0} 
\definecolor{darkblue}{rgb}{0,0,0.6} 

\definecolor{yellow}{rgb}{0.9,0.3,0}
\definecolor{dgreen}{rgb}{0,0.4,0}
\definecolor{dblue}{rgb}{0,0,0.6}
\definecolor{dred}{rgb}{0.8,0.0,0.1}
\definecolor{dgold}{rgb}{0.5,0.3,0.0}
\definecolor{dvio}{rgb}{0.6,0.3,0.5}
\definecolor{gray}{rgb}{0.5,0.5,0.5}
\definecolor{dbraun}{rgb}{.4,0.0,0.0}

\begin{document}
\thispagestyle{empty}
\large
 \begin{center}
\textbf{DISCONNECTION BY LEVEL SETS OF THE DISCRETE GAUSSIAN FREE FIELD AND ENTROPIC REPULSION}\\ 

\hspace{10pt}

\large
Maximilian Nitzschner\\

\hspace{10pt}

\hspace{15pt}
 \\ \textbf{ }
 \\  \textbf{ }
 \\
\textbf{Abstract}
\end{center}
\normalsize

 We derive asymptotic upper and lower bounds on the large deviation probability that the level set of the Gaussian free field on $\mathbb{Z}^d$, $d \geq 3$, below a level $\alpha$ disconnects the discrete blow-up of a compact set $A$ from the boundary of the discrete blow-up of a box that contains $A$, when the level set of the Gaussian free field above $\alpha$ is in a strongly percolative regime. These bounds substantially strengthen the results of \cite{S15}, where $A$ was a box and the convexity of $A$ played an important role in the proof. We also derive an asymptotic upper bound on the probability that the average of the Gaussian free field well inside the discrete blow-up of $A$ is above a certain level when disconnection occurs. 
 
 The derivation of the upper bounds uses the solidification estimates for porous interfaces that were derived in the work \cite{NS17} of A.-S. Sznitman and the author to treat a similar disconnection problem for the vacant set of random interlacements.
If certain critical levels for the Gaussian free field coincide, an open question at the moment, the asymptotic upper and lower bounds that we obtain for the disconnection probability match in principal order, and conditioning on disconnection lowers the average of the Gaussian free field well inside the discrete blow-up of $A$, which can be understood as entropic repulsion.
\\
\\
\\
\\
\\ \\
\\
\\
\\
\\
\\
\\
\\
\\
\\
Departement Mathematik \\
ETH Zürich \\
Rämistrasse 101 \\
8092 Zürich, Switzerland
\newpage \textbf{ }
\thispagestyle{empty} \newpage
\setcounter{page}{1}
\section{Introduction}
 \selectlanguage{english}
Level sets of the Gaussian free field on $\mathbb{Z}^d$, $d \geq 3$, provide an example of a percolation model with long-range dependence. Its study goes back at least to the eighties, see \cite{BLM87}, \cite{LS86}, \cite{MS83}, and it has attracted considerable attention recently, see for instance  \cite{G04}, \cite{PR15}, \cite{RS13}, \cite{S15} and \cite{DPR17}. It is well-known that the model undergoes a phase transition at some critical level $h_\ast(d)$, which is both finite (see \cite{RS13}) and strictly positive (as established recently in \cite{DPR17}) in all dimensions $d \geq 3$. \\

Remarkably, some results for the level-set percolation of the Gaussian free field have been obtained by applying methods from the study of random interlacements, a model of random subsets of $\mathbb{Z}^d$ introduced in \cite{S10} with similar percolative properties. In the present work, we obtain large deviation upper and lower bounds on the probability that the discrete blow-up of a compact set $A \subseteq \mathbb{R}^d$ gets disconnected from the boundary of the discrete blow-up of a box containing $A$ in its interior, by the level set of a Gaussian free field below a level $\alpha$, when the level set \textit{above} $\alpha$ is in a strongly percolative regime.  Disconnection problems of this type have been studied for the vacant set of random interlacements in \cite{LS13} (concerning lower bounds) and more recently in \cite{NS17} (concerning upper bounds), improving on an earlier result from \cite{S17}. The results of this article substantially generalize the findings of \cite{S15}, where large deviation upper and lower bounds on the disconnection probability described above were derived for the case where $A$ was itself a box in $\mathbb{R}^d$, and the convexity of $A$ played an important role in the proof. \\

 While the generalization of the lower bound is straightforward, the proof of the upper bound requires an improved coarse-graining procedure to handle situations in which $A$ is not convex. Specifically, we make use of the solidification estimates for porous interfaces of \cite{NS17}, where a uniform lower bound on the capacity of 'porous deformations' of arbitrary compact sets with positive capacity is given. It is plausible, but open at the moment, that the asymptotic upper and lower bounds we obtain actually match. This would result from showing the equality of $h_{\ast}$ and two other critical levels $\overline{h}$ and $h_{\ast\ast}$ that we introduce below. We also consider the event that the average of the Gaussian free field well inside the discrete blow-up of $A$ is above $-(\overline{h}-\alpha) + \Delta$ for some $\Delta > 0$ and obtain an asymptotic large deviation upper bound on the probability of the intersection of this event with the disconnection event described above. If the three critical levels $\overline{h},h_\ast$ and $h_{\ast\ast}$ match, our bounds underpin a downward shift of this average that can be understood as \textit{entropic repulsion}.
\\ 

We will now describe the model and our results in more detail. Consider $\mathbb{Z}^d$, $d \geq 3$, and let $\pM$ be the law on $\mathbb{R}^{\mathbb{Z}^d}$ under which the canonical process $\varphi = (\varphi_x)_{x \in \mathbb{Z}^d}$ is a Gaussian free field, that is 
\begin{equation}
\label{eq:GFFDef}
\begin{split}
& \text{under }\pM, \ \varphi \text{ is a centered Gaussian field on }\mathbb{Z}^d \text{ with }  \\
&\mathbb{E}[\varphi_x \varphi_y] = g(x,y) \text{ for all }x,y \in \mathbb{Z}^d,
\end{split}
\end{equation}
 where $g(\cdot,\cdot)$ denotes the Green function of the simple random walk, defined below in \eqref{eq:GreenFunction}. For $\alpha \in \mathbb{R}$, we define the level set (or excursion set) above $\alpha$ as 
 \begin{equation}
 E^{\geq \alpha} = \left\lbrace x \in \mathbb{Z}^d ; \varphi_x \geq \alpha \right\rbrace,
 \end{equation}
 which is a random subset of $\mathbb{Z}^d$. For any $d \geq 3$, one defines a critical level $h_\ast = h_\ast(d) \in (0,\infty)$ as 
 \begin{equation}
 h_\ast(d) = \inf \left\lbrace \alpha \in \mathbb{R} ; \pM\left[ 0 \stackrel{\geq \alpha}{\longleftrightarrow}\infty \right] = 0 \right\rbrace,
 \end{equation}
 with $\left\lbrace x \stackrel{\geq \alpha}{\longleftrightarrow}\infty  \right\rbrace$ denoting the event that $x \in \mathbb{Z}^d$ belongs to an infinite connected component of $E^{\geq \alpha}$. To study disconnection problems, one can introduce a second critical parameter $h_{\ast\ast} = h_{\ast\ast}(d) (\geq h_\ast)$ via
 \begin{equation}
 \label{eq:h_astast}
 h_{\ast\ast}(d) = \inf \left\lbrace \alpha \in \mathbb{R} ; \liminf_L \pM\left[ B_L \stackrel{\geq \alpha}{\longleftrightarrow} \partial B_{2L} \right] = 0 \right\rbrace,
 \end{equation}
 where $B_L$ is the closed sup-norm ball in $\mathbb{Z}^d$ with radius $L$, centered at the origin, $\partial B_{2L}$ is the (external) boundary of $B_{2L}$, and $\left\lbrace  B_L \stackrel{\geq \alpha}{\longleftrightarrow} \partial B_{2L}  \right\rbrace$ stands for the event that there is a nearest-neighbor path in the excursion set $E^{\geq \alpha}$ connecting $B_L$ and $\partial B_{2L}$. One can show (see \cite{PR15}, \cite{RS13}) that $h_{\ast\ast} < \infty$ for all $d \geq 3$ and that for $\alpha > h_{\ast\ast}$ the connectivity function $\pM\left[x\stackrel{\geq \alpha}{\longleftrightarrow}y \right]$ (with hopefully obvious notation) has a stretched exponential decay in $|x-y|$, the Euclidean distance between $x, y \in \mathbb{Z}^d$, which actually, is an exponential decay when $d \geq 4$. The range $\alpha>  h_{\ast\ast}$ is called the strongly non-percolative regime, and it is conjectured that in fact $h_\ast = h_{\ast\ast}$. Recently developed techniques stemming from the application of the theory of randomized algorithms to percolation models (see \cite{D-CRT17}) may provide a step towards a proof of this conjecture. \\ 
 
 Consider a non-empty compact set $A \subseteq \mathbb{R}^d$ contained in the interior of a box of side-length $2M$, $M > 0$, centered at the origin, which we call $B_{\mathbb{R}^d}(0,M)$, i.e. 
 \begin{equation}
 \label{eq:ContainedInBall}
 A \subseteq \mathring{B}_{\mathbb{R}^d}(0,M).
 \end{equation}
 We define the discrete blow-up of $A$ and the boundary of the discrete blow-up of the box that encloses $A$ as
 \begin{equation}
 \label{eq:BlowUps}
 A_N = (NA) \cap \mathbb{Z}^d, \qquad S_N = \lbrace x \in \mathbb{Z}^d ; |x|_\infty = [MN]\rbrace,
 \end{equation}
 where we denote by $|\cdot |_\infty$ the sup-norm and by $[\cdot]$ the integer part of a real number. We are interested in the large-$N$ behavior of the disconnection event
 \begin{equation}
 \label{eq:DisconnectionEvent}
 \mathcal{D}^\alpha_N = \left\lbrace A_N \stackrel{\geq \alpha}{\nleftrightarrow} S_N \right\rbrace,
 \end{equation}
 corresponding to the absence of a nearest-neighbor path in $E^{\geq \alpha}$ starting in $A_N$ and ending in $S_N$. Our main motivation is to investigate the large $N$ behavior of $\pM[\mathcal{D}^\alpha_N]$ and of the conditional probability measure $\pM[\ \cdot \ | \mathcal{D}^\alpha_N]$. \\
 
 As far as lower bounds on $\pM[\mathcal{D}^\alpha_N]$ are concerned, we show in Theorem \ref{Theorem21} that for $\alpha \leq h_{\ast\ast}$, 
 \begin{equation}
 \label{eq:LowerBound}
 \liminf_N \frac{1}{N^{d-2}} \log \pM[\mathcal{D}^\alpha_N] \geq - \frac{1}{2d} (h_{\ast\ast}-\alpha)^2\text{cap}(A),
 \end{equation}
 with $\text{cap}(\cdot)$ denoting the Brownian capacity, see e.g. p.58 of \cite{PS78}. The proof uses the change of probability method and is a straightforward generalization of the proof of Theorem 2.1 of \cite{S15}. We mention at this point that a result similar to \eqref{eq:LowerBound} holds for the vacant set of random interlacements below a certain level, which was derived in \cite{LS13} using so-called 'tilted interlacements'. \\ 
 
 Let us now turn to the (more delicate) upper bounds. One introduces another critical level $\overline{h}$ such that the region $\alpha < \overline{h}$ corresponds to a strongly percolative regime for $E^{\geq \alpha}$, see (5.3) of \cite{S15} for its precise definition. It can be shown that $\overline{h}\leq h_\ast (\leq h_{\ast\ast})$, see Remark 5.1 of \cite{S15}. Apart from $\overline{h} > -\infty$, not much is known about $\overline{h}$, not even that $\overline{h}\geq 0$. It is plausible, but presently open that $\overline{h} = h_\ast = h_{\ast\ast}$.
 For $\alpha < \overline{h}$, as shown in Theorem \ref{Theorem31}, 
  \begin{equation}
 \label{eq:UpperBound}
 \limsup_N \frac{1}{N^{d-2}} \log \pM[\mathcal{D}^\alpha_N] \leq - \frac{1}{2d} (\overline{h}-\alpha)^2\text{cap}(\mathring{A}),
 \end{equation}
 where $\mathring{A}$ denotes the interior of $A$. If $\overline{h} = h_\ast = h_{\ast\ast}$ holds, then \eqref{eq:LowerBound} and \eqref{eq:UpperBound} yield in the percolative case $\alpha < h_\ast$ asymptotic upper and lower bounds for $\pM[\mathcal{D}^\alpha_N]$ that match in principal order, if $A$ is regular in the sense that $\text{cap}(A) = \text{cap}(\mathring{A})$. \\

As mentioned before, \eqref{eq:UpperBound} is a genuine generalization of Theorem 5.5 of \cite{S15} to non-convex sets $A$. It also has an analogue in the theory of random interlacements, see Theorem 4.1 of \cite{NS17}. \\

In addition to the upper bound \eqref{eq:UpperBound}, 
we show in Theorem \ref{Theorem43} that for any non-empty open subset $V$ of $\mathring{A}$ with $\overline{V} \subseteq \mathring{A}$ and any $\beta > 0$, one has in the region $\alpha < \overline{h}$:
 \begin{equation}
 \label{eq:DisconnectionAndPositive}
\begin{split}
\limsup_N & \frac{1}{N^{d-2}}\log\pM\bigg[ \bigg\lbrace \frac{1}{|V_N|} \sum_{x \in V_N} \varphi_x \geq -(\overline{h}-\alpha) + \Delta \bigg\rbrace \cap \mathcal{D}^\alpha_N \bigg] \\ 
& \qquad\leq - \frac{1}{2d}\left( \overline{h} - \alpha + \beta \Delta\right)^2 \frac{\text{\normalfont{cap}}(\mathring{A})}{1+ \beta^2 c_6(V,M)},
\end{split}
 \end{equation} 
 where $V_N = (NV) \cap \mathbb{Z}^d$, $|\cdot|$ denotes the cardinality of a subset of $\mathbb{Z}^d$, and $c_6(V,M)$ is an explicit constant depending only on $d$, $V$ and $M > 0$, see \eqref{eq:BoundZf3}. Importantly, $\beta > 0$ can be chosen small enough such that the right-hand side of \eqref{eq:DisconnectionAndPositive} is strictly smaller than the the right-hand side of \eqref{eq:UpperBound}. If $\overline{h} = h_\ast = h_{\ast\ast}$ and $A$ is regular in the sense that $\text{cap}(A) = \text{cap}(\mathring{A})$, one can infer from this and the disconnection lower bound \eqref{eq:LowerBound} a behavior reminiscent of \textit{entropic repulsion}: for any non-empty open $V$ with $\overline{V} \subseteq \mathring{A}$, conditionally on disconnection, the average of the Gaussian free field over $V_N$ experiences with high probability a push down by an amount $-(h_\ast - \alpha)$, in the sense that:
 \begin{equation}
 \label{eq:EntropicRep}
 \lim_N \pM\left[ \frac{1}{|V_N|}\sum_{x \in V_N} \varphi_x  < -(h_\ast - \alpha) + \Delta \bigg\vert \mathcal{D}^\alpha_N \right] = 1
\end{equation}  
for any $\Delta > 0$ and $\alpha < h_\ast$, see Corollary \ref{LastCor}. \\ \\
Classically, entropic repulsion results for the discrete Gaussian free field involve conditioning on the event that $\varphi$ is non-negative over a subset of $\mathbb{Z}^d$ that can be written as the discrete blow-up of a set in $\mathbb{R}^d$. For instance, if $D_N = (ND) \cap \mathbb{Z}^d$ with $D$ a box in $\mathbb{R}^d$, it is known that conditionally on $\lbrace \varphi_x \geq 0 \text{ for all }x \in D_N \rbrace$, the average of the Gaussian free field over $D_N$ is repelled to a level $\sqrt{4g(0,0)\log N}$ for large $N$, see e.g. Theorem 3.1, (3) of \cite{Giacomin}. In contrast to that, it seems that conditioning on disconnection lowers the average of the Gaussian free field inside $V_N$ by a constant value that depends on $\alpha$. \\

Let us now describe the structure of this article. In Section 1, we recall some basic notation and results concerning random walks, the Gaussian free field, the change of probability method, and we state the capacity lower bound of \cite{NS17} that will play a pivotal role in the proof of our main result, Theorem \ref{Theorem31}, as well as the solidification estimates of \cite{NS17}. In Section 2, we state and prove Theorem \ref{Theorem21}, which corresponds to the large deviation lower bound \eqref{eq:LowerBound}. In Section 3, we recall some facts about the percolative properties of the level sets of the Gaussian free field, mainly from \cite{S15}, and then proceed to state and prove Theorem \ref{Theorem31}, which corresponds to \eqref{eq:UpperBound}. Finally, in Section 4, we prove the asymptotic upper bound \eqref{eq:DisconnectionAndPositive}, see Theorem \ref{Theorem43}, and show in Corollary \ref{LastCor} that under the assumption $\overline{h} = h_\ast = h_{\ast\ast}$, \eqref{eq:EntropicRep} holds true for any $\Delta > 0$. \\

Finally, our convention on constants is that $c,c',...$ denote generic positive constants that may change from place to place, and depend only on the dimension $d$. Numbered constants ($c_1,c_2,...$) refer to the value assigned to them the first time they appear in the text. We indicate dependence on additional parameters in the notation. 
\section{Notation and useful results}
In this section we introduce further notation and review some results concerning random walks, potential theory, the discrete Gaussian free field and the notion of relative entropy. We will also include the solidification estimates and a related lower bound on the Brownian capacity of porous interfaces from \cite{NS17}. The latter will play a crucial role in the derivation of the large deviation upper bound \eqref{eq:UpperBound} in Section 3. Throughout the entire article, we tacitly assume that $d \geq 3$.\\

 We begin by introducing some notation. For a real number $s$, we write $[s]$ for its integer part. We consider on $\mathbb{R}^d$ the Euclidean and $\ell^\infty$-norms $|\cdot|$ and $|\cdot|_\infty$ and denote for $r \geq 0$ by $B(x,r) = \lbrace y \in \mathbb{Z}^d; |x-y|_\infty \leq r \rbrace$ the closed $\ell^\infty$-ball of radius $r$ in $\mathbb{Z}^d$, centered at $x \in \mathbb{Z}^d$. For subsets $G,H \subseteq \mathbb{R}^d$, we denote by $d(G,H)$ their mutual $\ell^\infty$-distance, i.e. $d(G,H) = \inf \lbrace |x-y|_\infty ; x \in G , y \in H \rbrace$ and write for simplicity $d(x,G)$ instead of $d(\lbrace x \rbrace, G)$ for $x \in \mathbb{R}^d$. For $K \subseteq \mathbb{Z}^d$, we let $|K|$ denote the cardinality of $K$ and we define the boundary $\partial K = \lbrace y \in \mathbb{Z}^d \setminus K ; \exists x \in K : |y-x| = 1 \rbrace$ and the internal boundary $\partial_i  K = \lbrace x\in K ; \exists y \in \mathbb{Z}^d \setminus K : |x-y|=1 \rbrace$. If $x,y \in \mathbb{Z}^d$ fulfill $|x - y| = 1$, we call them neighbors and write $x \sim y$. We call $\pi : \lbrace 0,..., N \rbrace \rightarrow \mathbb{Z}^d$ a nearest neighbor path (of length $N+1 \geq 2$) if $\pi(i) \sim \pi(i+1)$ for all $0 \leq i \leq N-1$. For subsets $K,K',U \subseteq \mathbb{Z}^d$, we write $K \stackrel{U}{\leftrightarrow} K'$ (resp. $K \stackrel{U}{\nleftrightarrow} K'$) if there is a path $\pi$ with values in $U$ starting in $K$ and ending in $K'$ (resp. if there is no such path) and we say that $K$ and $K'$ are connected in $U$ (resp. not connected in $U$).\\

We will now state some well-known facts about the discrete time simple random walk on $\mathbb{Z}^d$. To this end, let $(X_n)_{n \geq 0}$ be the canonical process on $(\mathbb{Z}^d)^{\mathbb{N}}$ and $P_x$ the canonical law of a simple random walk on $\mathbb{Z}^d$ started at $x \in \mathbb{Z}^d$. We denote the corresponding expectations by $E_x$ and define the Green function $g(\cdot,\cdot)$ of the walk,
\begin{equation}
\label{eq:GreenFunction}
g(x,y) = \sum_{n \geq 0} P_x[X_n = y], \text{ for } x,y \in \mathbb{Z}^d,
\end{equation}
which is finite (due to transience), symmetric, and fulfills $g(x,y) = g(x-y,0) \stackrel{\text{def}}{=} g(x-y)$. \\It is known that (see e.g. Theorem 5.4, p.31 of \cite{L91}) 
\begin{equation}
g(x) \sim \frac{C_d}{|x|^{d-2}}, \qquad \text{as } |x| \rightarrow \infty, \text{ with }C_d = \frac{d}{2\pi^{\frac{d}{2}}}\Gamma\left(\frac{d}{2}-1 \right).
\end{equation}
Given $K \subseteq \mathbb{Z}^d$, we also introduce stopping times (with respect to the canonical filtration generated by $(X_n)_{n \geq 0})$ $H_K = \inf \lbrace n \geq 0; X_n \in K \rbrace$, $\widetilde{H}_K = \inf \lbrace n \geq 1; X_n \in K \rbrace$, and $T_K = \inf \lbrace n \geq 0; X_n \notin K \rbrace$, the entrance, hitting and exit times of $K$. \\

We will now discuss some aspects of potential theory associated with the simple random walk. For a finite subset $K \subseteq \mathbb{Z}^d$, we denote by
\begin{equation}
e_K(x) = P_x[\widetilde{H}_K = \infty] \mathbbm{1}_K(x), \text{ for } x \in \mathbb{Z}^d,
\end{equation}
the equilibrium measure of $K$ and its total mass
\begin{equation}
\text{cap}_{\mathbb{Z}^d}(K) = \sum_{x \in K}e_K(x),
\end{equation}
the capacity of $K$. We also recall that for finite $K \subseteq \mathbb{Z}^d$, one has
\begin{equation}
\label{eq:EquilibriumPotential}
P_x[H_K < \infty] = \sum_{x' \in K} g(x,x') e_K(x'), \text{ for }x \in \mathbb{Z}^d,
\end{equation} 
see e.g.\ Theorem T1, p.300 of \cite{Sp13}.
For a box of size $L$, which we denote by $B_L = B(0,L)$, one knows that (see e.g. (2.16), p.53 of \cite{L91})
\begin{equation}
\label{eq:BoxCapBound}
cL^{d-2} \leq \text{cap}_{\mathbb{Z}^d}(B_L) \leq c' L^{d-2}, \text{ for } L \geq 1.
\end{equation}
An important quantity in potential theory is the Dirichlet form $\mathcal{E}_{\mathbb{Z}^d}(\cdot,\cdot)$, which is defined for finitely supported functions $f: \mathbb{Z}^d \rightarrow \mathbb{R}$ as 
\begin{equation}
\label{eq:DirichletForm}
\mathcal{E}_{\mathbb{Z}^d}(f,f) = \frac{1}{4d}\sum_{x \sim y}(f(x) - f(y))^2,
\end{equation}
and by polarization this definition can be extended to 
\begin{equation}
\mathcal{E}_{\mathbb{Z}^d}(f,g) = \frac{1}{4d} \sum_{x\sim y}(f(x) - f(y))(g(x)-g(y)),
\end{equation}
if at least one of the functions $f,g : \mathbb{Z}^d \rightarrow \mathbb{R}$ has finite support. The capacity and the Dirichlet form $\mathcal{E}_{\mathbb{Z}^d}(\cdot,\cdot)$ are related via $\text{cap}_{\mathbb{Z}^d}(K) = \inf_f \mathcal{E}_{\mathbb{Z}^d}(f,f)$, where $f$ runs over the set of finitely supported functions with a restriction to $K$ bigger or equal to $1$, see e.g. (2.10), p.18 of \cite{W00}. \\
We now turn to the Gaussian free field on $\mathbb{Z}^d$, $d \geq 3$. Recall the definitions of $(\varphi_x)_{x \in \mathbb{Z}^d}$ and $\pM$ from \eqref{eq:GFFDef}. For any finitely supported $f: \mathbb{Z}^d \rightarrow \mathbb{R}$, one has 
\begin{equation}
\label{eq:ExpectationDirichletForm}
\begin{split}
i) \ &\mathbb{E}[\mathcal{E}(f,\varphi)\varphi_x] = f(x), \text{ for } x \in \mathbb{Z}^d \\
ii) \ &\mathbb{E}[\mathcal{E}(f,\varphi)^2] = \mathcal{E}(f,f),
\end{split}
\end{equation}
see (1.14) of \cite{S15}. For $U \subseteq \mathbb{Z}^d$ finite, one furthermore defines the harmonic average $h^U$ of $\varphi$ in $U$ and the local field $\psi^U$, via 
\begin{equation}
\label{eq:HarmonicAverage}
h^U_x = E_x[\varphi_{X_{T_U}}] = \sum_{y \in \mathbb{Z}^d} P_x[X_{T_U} = y]\varphi_y, \text{ for } x \in \mathbb{Z}^d,
\end{equation}
\begin{equation}
\label{eq:LocalField}
\psi^U_x = \varphi_x - h^U_x, \text{ for } x\in \mathbb{Z}^d.
\end{equation}
Obviously, $\varphi_x = h^U_x + \psi^U_x$ and $\psi_x^U = 0$ if $x \in \mathbb{Z}^d \setminus U$, whereas $h^U_x = \varphi_x$ in that case. The 'domain Markov property' of the Gaussian free field asserts that
\begin{equation}
\begin{split}
&(\psi^U_x)_{x \in \mathbb{Z}^d} \text{ is independent of }\sigma(\varphi_y; y \in U^c) \text{ (in particular of }(h^U_x)_{x \in \mathbb{Z}^d}), \\
&\text{and is distributed as a centered Gaussian field with covariance } g_U(\cdot,\cdot),
\end{split}
\end{equation}
(where, $g_U(\cdot,\cdot)$ is the Green function of the random walk killed upon exiting $U$, see (1.3) of \cite{S15}). \\

Recall that for $h_{\ast\ast}$ from \eqref{eq:h_astast}, one has $0 < h_\ast \leq h_{\ast\ast} < \infty$. Above $h_{\ast\ast}$, the probability that there is a path in $E^{\geq \alpha}$ connecting $0$ to $\partial B_L$ has a stretched exponential decay in $L$: 
\begin{equation}
\label{eq:StretchedExponentialDecay}
\pM[0 \stackrel{\geq \alpha}{\leftrightarrow} \partial B_L] \leq c_1(\alpha) e^{-c_2(\alpha)L^{c_3}}, \text{ for } L \geq 0, \alpha > h_{\ast\ast},
\end{equation}
and in fact $c_3 = 1$ when $d \geq 4$ (and $c_3 < 1$ due to a logarithmic correction when $d = 3$), see \cite{PR15}. This property plays a pivotal role in the proof of the lower bound \eqref{eq:LowerBound} in the next section. \\

We will also need a classical inequality involving the relative entropy of two probability measures. For a probability measure $\widetilde{Q}$ absolutely continuous with respect to $Q$, the relative entropy of $\widetilde{Q}$ with respect to $Q$ is defined as 
\begin{equation}
H(\widetilde{Q}|Q) = E^{\widetilde{Q}}\left[\log \tfrac{d\widetilde{Q}}{dQ} \right] = E^Q\left[\tfrac{d\widetilde{Q}}{dQ} \log \tfrac{d\widetilde{Q}}{dQ} \right] \in [0, \infty],
\end{equation}
where $E^Q$ and $E^{\widetilde{Q}}$ denote the expectations associated to $Q$ and $\widetilde{Q}$, respectively. If $C$ is an event with $\widetilde{Q}[C] > 0$, one has 
\begin{equation}
\label{eq:EntropyBound}
Q[C] \geq \widetilde{Q}[C] \exp\left(- \frac{1}{\widetilde{Q}[C]}\left(H(\widetilde{Q}|Q) + \frac{1}{e}\right)\right),
\end{equation}
see e.g. p.76 of \cite{DS89}. \\

We conclude this section by recalling a uniform lower bound on the Brownian capacity of 'porous interfaces' surrounding a compact set $A \subseteq \mathbb{R}^d$, as well as an asymptotic lower bound on the trapping probability of Brownian motion by such porous interfaces, both from \cite{NS17}. Consider a non-empty bounded Borel subset $U_0$ of $\mathbb{R}^d$ with complement $U_1 = \mathbb{R}^d \setminus U_0$ and boundary $S= \partial U_0 = \partial U_1$. For $x \in \mathbb{R}^d$ and $l \in \mathbb{Z}$, let 
\begin{equation}
\widehat{\sigma}_l(x) = \frac{|B_{\mathbb{R}^d}(x,2^{-l})\cap U_1|}{|B_{\mathbb{R}^d}(x,2^{-l})|}
\end{equation}
 be the local density function associated to $U_1$, where $| \cdot |$ stands for the Lebesgue measure on $\mathbb{R}^d$ and $B_{\mathbb{R}^d}(x,r) = \lbrace y \in \mathbb{R}^d; |x-y|_\infty \leq r \rbrace$ for the closed ball around $x \in \mathbb{R}^d$ in $| \cdot |_\infty$-norm with radius $r \geq 0$. Moreover, for $l_\ast \geq 0$ and a non-empty compact subset $A$ of $\mathbb{R}^d$, introduce 
 \begin{equation}
 \begin{split}
 \mathcal{U}_{l_\ast, A} = \ &\text{the collection of bounded Borel subsets }U_0 \subseteq \mathbb{R}^d \text{ with } \\
 & \widehat{\sigma}_l(x) \leq \frac{1}{2} \text{ for all }x\in A \text{ and } l \geq l_\ast.
 \end{split}
\end{equation}  
Furthermore, denote by $W_z$ the Wiener measure started at $z \in \mathbb{R}^d$ under which the canonical process $(Z_t)_{t \geq 0}$ on $C(\mathbb{R}_+, \mathbb{R}^d)$ is a Brownian motion on $\mathbb{R}^d$, starting from $z \in \mathbb{R}^d$. Similarly as for the random walk, one can introduce stopping times (with respect to the canonical filtration generated by $(Z_t)_{t \geq 0}$) $H_B = \inf \lbrace s \geq 0; Z_s \in B \rbrace$ for $B \subseteq \mathbb{R}^d$ closed, the entrance time of Brownian motion into $B$, and $T_U = \inf \lbrace s \geq 0; Z_s \notin U \rbrace ( = H_{U^c})$ for $U \subseteq \mathbb{R}^d$ open, the exit time of Brownian motion from $U$. Moreover, one defines the first time when $Z$ moves at $|\cdot|_\infty$-distance $ r\geq 0$ from its starting point (which is also a stopping time), 
\begin{equation}
\tau_r = \inf \lbrace s \geq 0; |Z_s - Z_0|_\infty \geq r \rbrace.
\end{equation}
With these definitions at hand, we can introduce for a given non-empty bounded Borel subset $U_0 \subseteq \mathbb{R}^d$, $\varepsilon > 0$ and $\eta \in (0,1)$ the class of 'porous interfaces' 
\begin{equation}
\label{eq:ClassofporousInterf}
\begin{split}
\mathscr{S}_{U_0,\varepsilon,\eta} &= \ \text{the class of }\Sigma \subseteq \mathbb{R}^d \text{ compact with } W_z[H_\Sigma < \tau_\varepsilon] \geq \eta, \text{ for all } z \in \partial U_0.
\end{split}
\end{equation}
Informally, $\varepsilon$ and $\eta$ correspond to the distance from $S = \partial U_0$, at which the porous interface $\Sigma$ is 'felt' and the strength of its presence, respectively. With $\text{cap}(\cdot)$ denoting the Brownian capacity (as in the introduction), one has for any compact subset $A$ of $\mathbb{R}^d$ such that $\text{cap}(A) > 0$ and $\eta \in (0,1)$ the uniform capacity lower bound
\begin{equation}
\label{eq:SolidificationEstimate}
\lim_{u \rightarrow 0} \inf_{\varepsilon \leq u2^{-{l_\ast}}} \inf_{U_0 \in \mathcal{U}_{l_\ast,A}} \inf_{\Sigma \in \mathscr{S}_{U_0,\varepsilon,\eta}} \frac{\text{cap}(\Sigma)}{\text{cap}(A)} = 1,
\end{equation}
see (3.15) of Corollary 3.4 in \cite{NS17}. This result will function as a subsitute for a Wiener-type criterion if $A$ is not convex with $\Sigma$ being constructed from a configuration of 'blocking boxes' present when the disconnection event $\mathcal{D}^\alpha_N$ occurs. \\

Finally, we will need in Section 4 that with the same notation as above, one has the solidification estimate for porous interfaces:
\begin{equation}
\label{eq:SolidificationProb}
\lim_{u \rightarrow 0} \sup_{\varepsilon \leq u2^{-l_\ast}} \sup_{U_0 \in \mathcal{U}_{l_\ast,A}} \sup_{\Sigma \in \mathscr{S}_{U_0,\varepsilon,\eta}} \sup_{x\in A} W_x[H_\Sigma = \infty] = 0,
\end{equation}
see (3.3) of Theorem 3.1 in \cite{NS17}.
\section{Disconnection lower bound}
In this section we consider levels $\alpha \leq h_{\ast\ast}$ and give an asymptotic lower bound on the disconnection probability $\pM[\mathcal{D}^\alpha_N]$ for large $N$ (see \eqref{eq:DisconnectionEvent} for the definition of $\mathcal{D}^\alpha_N$). The proof is a straightforward extension of the one given in Section 2 of \cite{S15} in the case when $A$ was the box $[-1,1]^N$, and uses a technical result from \cite{LS13}, where the analogous statement in the case of disconnection by random interlacements was given. \\

Consider a non-empty compact subset $A$ of $\mathbb{R}^d$ that fulfills \eqref{eq:ContainedInBall} and recall the notation from \eqref{eq:BlowUps}.
\begin{theorem}
\label{Theorem21}
Assume that $\alpha \leq h_{\ast\ast}$, with $h_{\ast\ast}$ the critical level from \eqref{eq:h_astast}. Then one has
\begin{equation}
\liminf_N \frac{1}{N^{d-2}} \log \pM[\mathcal{D}^\alpha_N] \geq - \frac{1}{2d}(h_{\ast\ast} - \alpha)^2 \text{\normalfont{cap}}(A).
\end{equation}
\end{theorem}
Before we start the proof of this Theorem, we state the auxiliary technical result alluded to above. 
 Recall the notation concerning Brownian motion above \eqref{eq:ClassofporousInterf}. For any $\eta \in (0,1)$, let $A^\eta$ be the closed $\eta$-neighborhood of $A$. Let $U$ be an open Euclidean ball in $\mathbb{R}^d$ with radius $r_U > 0$, centered at the origin, such that one has 
 \begin{equation}
 A^{2\delta} \subseteq \mathring{B}_{\mathbb{R}^d}(0,M) \subseteq U.
 \end{equation}
 We define the function $h$ from $\mathbb{R}^d$ to $[0,1]$ by 
 \begin{equation}
 h(z) = W_z[H_{A^{2\delta}} < T_U], \text{ for } z \in \mathbb{R}^d,
 \end{equation}
 which fulfills $h(z) = 1$ for $z \in A^{2\delta}$ ($h$ is the equilibrium potential of $A^{2\delta}$ relative to $U$). For $\vartheta \in (0, \delta)$, let $\phi^\vartheta$ be a smooth non-negative function with $\text{supp}\ \phi^\vartheta$ contained in the Euclidean ball in $\mathbb{R}^d$ with radius $\vartheta$, centered at the origin, and $\int \phi^\vartheta(z)dz = 1$, and write 
 \begin{equation}
 h^\vartheta = h \ast \phi^\vartheta
 \end{equation}
 for the convolution of $h$ and $\phi^\vartheta$. Note that $h^\vartheta = 1$ on $A^{2\delta - 2 \vartheta} \supseteq A$ by construction. Finally, we define the restriction to $\mathbb{Z}^d$ of the blow-up of $h$ as 
 \begin{equation}
 h_N(x) = h^\vartheta\left(\frac{x}{N}\right), \text{ for } x \in \mathbb{Z}^d.
 \end{equation}
 Then Proposition 2.4 of \cite{LS13} states that
 \begin{equation}
 \label{eq:AuxiliaryResultLS}
 \lim_{\delta \rightarrow 0} \lim_{r_U \rightarrow \infty} \lim_{\vartheta \rightarrow 0} \lim_{N \rightarrow \infty} \frac{1}{N^{d-2}} \mathcal{E}_{\mathbb{Z}^d}(h_N,h_N) = \frac{1}{d}\text{cap}(A),
 \end{equation}
 with $\mathcal{E}_{\mathbb{Z}^d}(\cdot,\cdot)$ the Dirichlet form introduced in \eqref{eq:DirichletForm}.
 \begin{proof}[Proof of Theorem \ref{Theorem21}]
 Consider for $f: \mathbb{Z}^d \rightarrow \mathbb{R}$ finitely supported the following 'tilted' probability measure on $\mathbb{R}^{\mathbb{Z}^d}$: 
 \begin{equation}
 \label{eq:TiltedPM}
 d\widetilde{\pM} = \exp\left\lbrace \mathcal{E}_{\mathbb{Z}^d}(f,\varphi) - \frac{1}{2}\mathcal{E}_{\mathbb{Z}^d}(f,f)\right\rbrace d\pM.
 \end{equation}
 By \eqref{eq:ExpectationDirichletForm} and Cameron-Martin's formula, $\widetilde{\pM}$ is indeed a probability measure and 
 \begin{equation}
 \label{eq:SameLaw}
 \begin{split}
 \varphi \text{ under }\widetilde{\pM}\text{ has the same law as }(\varphi_x + f(x))_{x \in \mathbb{Z}^d} \text{ under }\pM,
 \end{split}
 \end{equation}
 (see also (2.3) and (2.4) of \cite{S15}). We choose $\varepsilon > 0$, and define $\widetilde{\pM}_N$ as the probability measure associated to $f_N = -(h_{\ast\ast} - \alpha + \varepsilon)h_N$ (in place of $f$) in \eqref{eq:TiltedPM}. Using the entropy inequality \eqref{eq:EntropyBound}, we find that 
 \begin{equation}
 \label{eq:ApplicationOfEntrIneq}
 \pM[\mathcal{D}^\alpha_N] \geq \widetilde{\pM}_N[\mathcal{D}^\alpha_N]\exp\left(-\frac{1}{\widetilde{\pM}_N[\mathcal{D}^\alpha_N]} \left(\frac{1}{2}\mathcal{E}_{\mathbb{Z}^d}(f_N,f_N) + \frac{1}{e}\right) \right),
 \end{equation}
 since $H(\widetilde{\pM}_N|\pM) = \widetilde{\mathbb{E}}_N\left[ \mathcal{E}_{\mathbb{Z}^d}(f_N,\varphi)\right] - \frac{1}{2}\mathcal{E}_{\mathbb{Z}^d}(f_N,f_N) \stackrel{\eqref{eq:SameLaw}}{=} \frac{1}{2}\mathcal{E}_{\mathbb{Z}^d}(f_N,f_N)$, where we denote by $\widetilde{\mathbb{E}}_N$ the expectation with respect to $\widetilde{\pM}_N$. Theorem \ref{Theorem21} will follow once we show that 
 \begin{equation}
 \label{eq:MainStepLowerBound21}
 \lim_N \widetilde{\pM}_N[\mathcal{D}^\alpha_N] = 1.
 \end{equation}
 Following a similar reasoning as in (2.9) and (2.10) of \cite{S15}, we have that
 \begin{equation}
 \label{eq:Prove211}
 \begin{split}
 \widetilde{\pM}_N[\mathcal{D}^\alpha_N] & \stackrel{\eqref{eq:SameLaw}}{=} \pM[A_N \stackrel{\geq \alpha- f_N}{\nleftrightarrow} S_N]\geq \pM[A_N \stackrel{\geq h_{\ast\ast} + \varepsilon}{\nleftrightarrow} \partial_i (A^{2\delta - 2\vartheta})_N],
 \end{split}
 \end{equation}
 where we used in the last step that $f_N = -(h_{\ast\ast} - \alpha + \varepsilon)$ on $(NA^{2\delta-2\vartheta}) \cap \mathbb{Z}^d = (A^{2\delta - 2\vartheta})_N$. However, we can see now that 
 \begin{equation}
 \label{eq:UpperBoundProof}
 \begin{split}
  \pM[A_N 
 \stackrel{\geq h_{\ast\ast}+\varepsilon }{\longleftrightarrow} \partial_i(A^{2\delta - 2\vartheta})_N ] & = \pM\left[ \bigcup_{x \in A_N} \left\lbrace x \stackrel{\geq h_{\ast\ast}+\varepsilon }{\longleftrightarrow} \partial_i(A^{2\delta - 2\vartheta})_N \right\rbrace \right] \\
 & \leq |A_N| \pM\left[0 \stackrel{\geq h_{\ast\ast}+\varepsilon }{\longleftrightarrow} \partial B_{[c' (\delta- \vartheta)N ]} \right] \rightarrow 0 \qquad \text{as }N \rightarrow \infty, 
 \end{split}
 \end{equation}
 where we used \eqref{eq:StretchedExponentialDecay} to bound the probability in the second line as well as the fact that $|A_N|$ grows polynomially in $N$. The claim \eqref{eq:MainStepLowerBound21} follows directly from a combination of \eqref{eq:Prove211} and \eqref{eq:UpperBoundProof}. We now take the logarithm of \eqref{eq:ApplicationOfEntrIneq}, divide by $N^{d-2}$ and take the limit $N \rightarrow \infty$ to arrive at
 \begin{equation}
 \liminf_N \frac{1}{N^{d-2}} \log \pM[\mathcal{D}^\alpha_N] \geq - \frac{1}{2}(h_{\ast\ast} - \alpha + \varepsilon)^2 \lim_N \frac{1}{N^{d-2}} \mathcal{E}_{\mathbb{Z}^d}(h_N,h_N).
 \end{equation}
 The proof is concluded by taking the limits $\vartheta \rightarrow 0$, $r_U \rightarrow \infty$ and $\delta \rightarrow 0$, using \eqref{eq:AuxiliaryResultLS} and finally taking $\varepsilon \rightarrow 0$.
  \end{proof}
 \begin{assumption}
 \label{Remark22}
 1) Let $\left\lbrace A_N \stackrel{\geq \alpha}{\nleftrightarrow} \infty \right\rbrace$ be the event that there is no infinite path in $E^{\geq \alpha}$ starting in $A_N$. Clearly, this event contains $\mathcal{D}^\alpha_N$ and we obtain as an immediate consequence of Theorem \ref{Theorem21} that
 \begin{equation}
 \liminf_N \frac{1}{N^{d-2}} \log \pM\left[ A_N \stackrel{\geq \alpha}{\nleftrightarrow} \infty \right] \geq -\frac{1}{2d}(h_{\ast\ast} - \alpha)^2 \text{cap}(A).
\end{equation}
2) The result of Theorem \ref{Theorem21} can be generalized as follows: Instead of $\mathcal{D}^\alpha_N$, consider the events $\widehat{\mathcal{D}}^\alpha_N$ with $A_N$ in \eqref{eq:DisconnectionEvent} replaced by $\widehat{A}_N = (NA(\epsilon_N)) \cap \mathbb{Z}^d$ and some sequence $\epsilon_N \rightarrow 0$, where $A({\epsilon_N})$ denotes the closed $\epsilon_N$-neighborhood in $|\cdot|_\infty$-norm of $A$. In fact, we can for large $N$ modify \eqref{eq:Prove211} such that $\widetilde{\pM}_N[\widehat{\mathcal{D}}^\alpha_N] \geq \pM[\widehat{A}_N \stackrel{\geq h_{\ast\ast} + \varepsilon/2}{\nleftrightarrow} \partial_i (A(\epsilon_N)^{2\delta - 2\vartheta})_N]$, because $f_N = -(h_{\ast\ast}-\alpha + \varepsilon)$ on $(A^{2\delta-2\vartheta})_N$ and thus must be $ \leq  -(h_{\ast\ast}-\alpha + \varepsilon/2)$ on $(A(\epsilon_N))^{2\delta-2\vartheta}$ for $N$ large enough. The upper bound \eqref{eq:UpperBoundProof} thus continues to hold and we obtain
\begin{equation}
\liminf_N \frac{1}{N^{d-2}} \pM[\widehat{\mathcal{D}}^\alpha_N] \geq -\frac{1}{2d}(h_{\ast\ast} -\alpha)^2 \text{cap}(A).
\end{equation}
\hfill $\square$
 \end{assumption}
 \section{Disconnection upper bound}
 In this section, we proceed to prove the main result of this article, the asymptotic upper bound on the disconnection probability \eqref{eq:UpperBound}. The general strategy of the proof we give here is essentially derived from the proof of the analogous statement for disconnection by random interlacements in \cite{NS17}. The notion of 'good($\alpha,\beta,\gamma$)'-boxes (see (2.11)-(2.13) of \cite{S17}) is replaced by the notion of '$\psi$-good boxes at levels $\gamma,\delta$' from \cite{S15}, whereas the condition $N_u(D_z) \leq \beta \text{cap}_{\mathbb{Z}^d}(D_z)$ (see (2.14) of \cite{S17}) roughly corresponds to the notion of '$h$-good boxes at level $a$' from \cite{S15}. \\
 
Whereas the upper bound in Theorem 5.5 of \cite{S15} dealt with the case where $A$ was a box, we are able to treat here the case of a general compact set $A$ and do not require any assumption of convexity for $A$. Thus, we consider again a compact subset $A$ of $\mathbb{R}^d$ that fulfills \eqref{eq:ContainedInBall} and recall the notation from \eqref{eq:BlowUps}. \\
 
 The main result of this section is: 
 \begin{theorem}
 \label{Theorem31}
 Assume that $\alpha < \overline{h}$. Then one has 
 \begin{equation}
 \limsup_N \frac{1}{N^{d-2}} \log \pM[\mathcal{D}^\alpha_N] \leq - \frac{1}{2d}(\overline{h}-\alpha)^2\text{\normalfont{cap}}(\mathring{A}).
 \end{equation}
 Here, $\overline{h}$ is the critical value defined in (5.3) of \cite{S15}.
 \end{theorem}
 Before starting the proof of Theorem \ref{Theorem31}, we introduce further notation and recall some more results on the Gaussian free field from \cite{S15}. \\
 \\
 We consider some positive integers $L \geq 1$ and $K \geq 100$ and introduce the lattice $\mathbb{L} = L\mathbb{Z}^d$. We define for $z \in \mathbb{L}^d$
 \begin{equation}
 \label{eq:Boxes}
 \begin{split}
 B_z & = z + [0,L)^d\cap \mathbb{Z}^d \subseteq D_z = z + [-3L,4L)^d \cap \mathbb{Z}^d \\
 & \subseteq U_z = z + [-KL+1,KL-1)^d \cap \mathbb{Z}^d.
 \end{split}
 \end{equation}
 Let $\varphi = \psi_z + h_z$, where $h_z = h^{U_z}$ and $\psi_z = \psi^{U_z}$, be the decomposition of $\varphi$ into the harmonic average and the local field in $U_z$, see \eqref{eq:HarmonicAverage} and \eqref{eq:LocalField}. We refer to (5.7), (5.8) of \cite{S15} for the precise definition of a $\psi$-good box $B_z$ at levels $\gamma > \delta$ (with $\delta < \gamma < \overline{h}$), which is otherwise called $\psi$-bad at levels $\gamma > \delta$.  Being $\psi$-good at levels $\gamma > \delta$ in essence means that the box $B_z$ fulfills two properties: The set $B_z \cap \lbrace x \in \mathbb{Z}^d; \psi_z(x) \geq \gamma \rbrace$ contains a connected component of diameter $\geq \frac{L}{10}$ and for every neighboring box $B_{z'}$ (with $z' \in \mathbb{L}$, $|z-z'| = L$), all connected components of $B_z \cap \lbrace x \in \mathbb{Z}^d; \psi_z(x) \geq \gamma \rbrace$ and $B_{z'} \cap \lbrace x \in \mathbb{Z}^d; \psi_{z'}(x) \geq \gamma \rbrace$ of diameter $\geq \frac{L}{10}$ must be connected in $D_z \cap \lbrace x \in \mathbb{Z}^d; \psi_z(x) \geq \delta \rbrace$. Moreover, we recall (5.9) of \cite{S15} for the notion of an $h$-good box $B_z$ at level $a > 0$ (which is otherwise called $h$-bad at level $a$), which means that $\inf_{D_z} h_z > -a$. \\
 
 In the following, we state three facts from \cite{S15}, that are going to act as analogues to (4.13), (4.14) and (4.16) of \cite{NS17} in the proof. 
 \\
 1) The following connectivity statement holds, cf. Lemma 5.3 of \cite{S15}: 
 \begin{equation}
 \label{eq:Connectivity}
 \begin{split}
 &\text{If }B_{z_i}, 0 \leq i \leq n \text{ is a sequence of neighboring $L$-boxes, i.e. } z_i, 0 \leq i \leq n, \\
 & \text{form a nearest-neighbor path in }\mathbb{L} = L\mathbb{Z}^d, \text{ all of which are $\psi$-good at levels $\gamma,\delta$}
 \\ &\text{and $h$-good at level $a$, then there is a path in } E^{\geq \delta - a}\cap \left(\bigcup_{i=0}^n D_{z_i} \right) \text{ starting in $B_{z_0}$} \\
 &\text{and ending in $B_{z_n}$.} 
 \end{split}
 \end{equation}
 2) The next statement will provide us with an upper bound on the probability that all boxes associated to a finite collection $\mathcal{C} \subseteq \mathbb{L}$ become $h$-bad. To this end, consider
 \begin{equation}
 \label{eq:ConditionsOnC}
 \mathcal{C} \text{ a non-empty, finite subset of $\mathbb{L}$ containing points at mutual $|\cdot|_\infty$-distance $\geq \overline{K}L$},
 \end{equation}
 where $\overline{K} = 2K+1$, and furthermore
 \begin{equation}
 \label{eq:DefC}
 C = \bigcup_{z \in \mathcal{C}}B_z.
 \end{equation}
 Corollary 4.4 of \cite{S15} gives us that 
 \begin{equation}
 \label{eq:ExponentialControl}
 \limsup_L \sup_{\mathcal{C}} \left\lbrace \log \pM\left[ \bigcap_{z \in \mathcal{C}} \lbrace \inf_{D_z} h_z \leq -a \rbrace \right] + \frac{1}{2}\left( a - \frac{c_4}{K}\sqrt{\frac{|\mathcal{C}|}{\text{cap}_{\mathbb{Z}^d}(C)}}\right)^2_+ \frac{\text{cap}_{\mathbb{Z}^d}(C)}{\alpha(K)} \right\rbrace \leq 0,
 \end{equation}
 with $\alpha(K) > 1$ and $\lim_K \alpha(K) = 1$, and $\sup_{\mathcal{C}}$ denoting the supremum over all finite collections $\mathcal{C} \subseteq \mathbb{L}$ satisfying \eqref{eq:ConditionsOnC}. \\
 3) Finally, we need a super-exponential bound on the probability that 'many boxes' within a certain range become $\psi$-bad. Namely, there exists a positive function $\rho$ depending on $L$ (as well as on $\gamma, \delta$ and $K$), with $\lim_L \rho(L) = 0$ such that with the definition
 \begin{equation}
 N_L = \frac{L^{d-1}}{\log L}, \text{ for } L > 1,
 \end{equation}
 we have the bound
 \begin{equation}
 \label{eq:SuperExponentialBound}
 \begin{split}
 \lim_L \frac{1}{N_L^{d-2}} \log \pM[&\text{there are at least $\rho(L)\left(N_L/L \right)^{d-1}$ columns in}\\
 &\text{direction $e$ containing a $\psi$-bad box at levels $\gamma > \delta$}] = -\infty,
 \end{split}
 \end{equation}
 for every $e$ in the canonical basis of $\mathbb{R}^d$,
 where (as below (4.16) of \cite{NS17}), for a vector $e$ of the canonical basis of $\mathbb{R}^d$, one defines a column in $[-N_L,N_L]^d$ as the collection of $L$-boxes that intersect $[-N_L,N_L]^d$ and have the same projection on the hyperplane  $\lbrace x \in \mathbb{Z}^d; x\cdot  e = 0 \rbrace$ of $\mathbb{Z}^d$. This result is a slight modification of (5.18) and Proposition 5.4 of \cite{S15}. \\
 
 Recall the definition of $\rho_\ast(u)$ in (4.17) of \cite{NS17} and choose a positive sequence $(\gamma_N)_{N \geq 1}$ satisfying i) - iv) of (4.18) of the same reference, in particular $\gamma_N \leq 1$ and $\gamma_N \rightarrow 0$ as $N \rightarrow \infty$. We will now specify $L$ to be a function of $N$ (and call this choice $L_0$ from now on). Thus we define
 \begin{equation}
 \label{eq:L0Def}
 L_0 = \left[ \left(\gamma_N^{-1} N\log N \right)^{\frac{1}{d-1}}\right], \qquad \widehat{L}_0 = 100d \left[ \sqrt{\gamma_N}N \right],
 \end{equation}
 as well as the lattices
 \begin{equation}
 \mathbb{L}_0 = L_0 \mathbb{Z}^d, \qquad \widehat{\mathbb{L}}_0 = \frac{1}{100d}\widehat{L}_0\mathbb{Z}^d \left(= \left[ \sqrt{\gamma_N}N \right] \mathbb{Z}^d \right).
 \end{equation}
 With these preparations, we can start the proof of Theorem \ref{Theorem31}. Let us give a short outline of the coarse-graining procedure that is used in the proof. \\
 
 For large $N$, on the event $\mathcal{D}^\alpha_N$, one can extract a certain interface of $L_0$-boxes, either $\psi$-bad at levels $\gamma > \delta$ or $h$-bad at level $a$, with $a + \alpha = \delta$, that 'blocks' the way between $A_N$ and the complement of $B(0,(M+1)N)$. This is a consequence of the connectivity statement \eqref{eq:Connectivity}. One can then track the presence of this interface of $L_0$-boxes by using nearly macroscopic boxes of size $\widehat{L}_0 \gg L_0$, and select a random subset of $\widehat{\mathbb{L}}_0$, where the blocking boxes have a non-degenerate presence within $B(x,\widehat{L}_0)$, $x \in \widehat{\mathbb{L}}_0$, cf. \eqref{eq:SelectionOfBoxes}. Upon discarding an event of super-exponentially decaying probability, we receive an interface of boxes of size $\widehat{L}_0$, such that in each box there will be a substantial presence of boxes $B_z$ of size $L_0$, all $\psi$-good at levels $\gamma > \delta$ and $h$-bad at level $a$ (we will only need the fact that they are all $h$-bad). The selection of these boxes involves the application of certain isoperimetric controls of \cite{DP96}. This gives a partition, with a small combinatorial complexity of order $\exp(o(N^{d-2}))$ of the event $\mathcal{D}^\alpha_N \setminus \mathcal{B}_N$, where $\mathcal{B}_N$ has for our purposes negligible probability, into events $\mathcal{D}_{N,\kappa}$, $\kappa \in \mathcal{K}_N$ (see \eqref{eq:CoarseGraningScheme}). Below we provide an illustration of the situation when $\mathcal{D}_{N,\kappa}$ occurs and the scales we are using. With the help of \eqref{eq:ExponentialControl}, we are able to reduce the derivation of the desired upper bound to finding a uniform lower bound on $\text{cap}_{\mathbb{Z}^d}(C)$, $C$ being the union of selected $L_0$-boxes, which is where the pivotal Corollary 3.4 of \cite{NS17} comes into play.
   \begin{figure}[h]\label{fig:scalesDisc}
  \centering
  \includegraphics[width=0.7\textwidth]{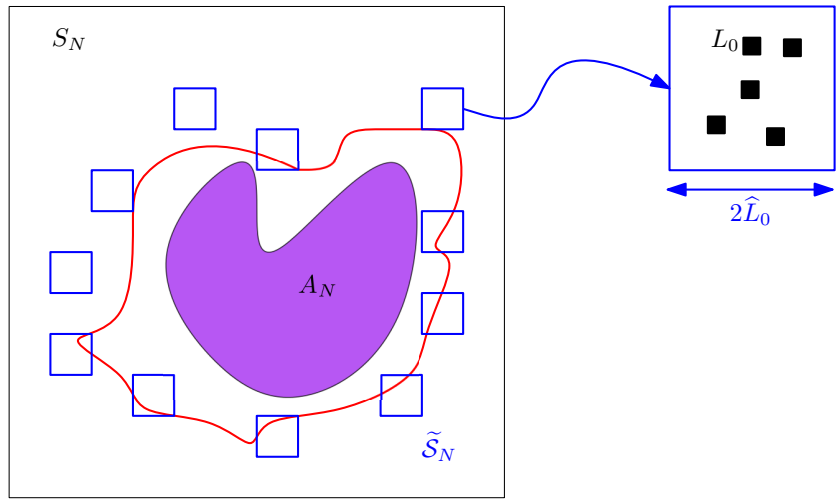} 
  \caption{Illustration of the situation when $\mathcal{D}_{N,\kappa}$ occurs, with the set of selected boxes of side-length $2\widehat{L}_0$ surrounding $A_N$ on the left (in blue), and selected $L_0$-boxes within one of these boxes (in black).}
 \end{figure}
 \begin{proof}[Proof of Theorem \ref{Theorem31}]
 In what follows, we assume
 \begin{equation}
 \label{eq:ConventionParameters}
 \alpha + a = \delta ( < \gamma < \overline{h}),
 \end{equation}
 and eventually we will let $\delta$ tend to $\overline{h}$.
 Recall the definitions of $L_0$ and $\widehat{L}_0$ from \eqref{eq:L0Def}. Without loss of generality, we assume that $\mathring{A} \neq \emptyset$ and $N \geq N_0(A,M)$ such that $A_N \subseteq B(0,MN) \setminus S_N$. We introduce the random subset 
 \begin{equation}
 \begin{split}
 \mathcal{U}^1 = & \text{ the union of all $L_0$-boxes $B_z$ either contained in $B(0,(M+1)N)^c$ or }\\
 & \text{ connected to an $L_0$-box in $B(0,(M+1)N)^c$ by a path of $L_0$-boxes $B_{z_i}$,}\\
 & \text{ $0 \leq i \leq n$, all (except possibly the last one) $\psi$-good at levels $\gamma > \delta$} \\
 & \text{ and  $h$-good at level $a > 0$.} 
 \end{split}
 \end{equation}
 One defines the local density
 \begin{equation}
 \widehat{\sigma}(x) = \frac{|\mathcal{U}^1 \cap B(x, \widehat{L}_0)|}{| B(x, \widehat{L}_0)|}, \qquad x \in \mathbb{Z}^d,
 \end{equation}
 which measures the presence of $\mathcal{U}^1$ within $B(x,\widehat{L}_0)$, and fulfills
 \begin{align}
 \widehat{\sigma}(x) & = 1,\text{ when } B(x,\widehat{L}_0 + L_0) \subseteq B(0,(M+1)N)^c \\ 
 \widehat{\sigma}(x) & = 0, \text{ when } B(x,\widehat{L}_0 + L_0) \subseteq A_N, \text{ for large enough }N \text{ on } \mathcal{D}^\alpha_N,
\end{align} see (4.30) and (4.31) of \cite{NS17}. We then introduce the random subset of $\widehat{\mathbb{L}}_0$, where the blocking boxes have a significant presence:
 \begin{equation}
 \label{eq:SelectionOfBoxes}
 \widehat{\mathcal{S}}_N = \left\lbrace x \in \widehat{\mathbb{L}}_0; \widehat{\sigma}(x) \in \left[\tfrac{1}{4}, \tfrac{3}{4}\right] \right\rbrace,
 \end{equation}
 as well as the compact subset $\Delta_N \subseteq \mathbb{R}^d$: 
 \begin{equation}
 \label{eq:CompactSubsetDelta}
 \Delta_N = \bigcup_{x \in \widehat{\mathcal{S}}_N } B_{\mathbb{R}^d}\left(\tfrac{x}{N}, \tfrac{1}{50d} \tfrac{\widehat{L}_0}{N} \right).
 \end{equation}
 $\Delta_N$ fulfills the 'insulation property', that is
 \begin{equation}
 \label{eq:InsulationProperty}
 \begin{split}
 & \qquad \qquad\qquad \widehat{\mathcal{S}}_N \subseteq B(0,(M+2)N) \cap \widehat{\mathbb{L}}_0, \text{ and} \\ 
 &\text{on $\mathcal{D}^\alpha_N$, the compact set $\lbrace z \in A; d(z,\partial A) \geq \tfrac{\widehat{L}_0 + L_0 + 1}{N}$ is }\\
 &\text{contained in the union of the bounded components of $\mathbb{R}^d \setminus \Delta_N \rbrace$}
 \end{split}
\end{equation}  (see (4.34), (4.35) of \cite{NS17}), the proof of which pertains to our case line by line. \\
 
 As a next step, we define the 'bad event'
 \begin{equation}
 \begin{split}
 \mathcal{B}_N = \bigcup_e & \lbrace \text{there are at least $\rho(L_0) (N_{L_0}/L_0)^{d-1}$ columns of $L_0$-boxes} \\
 & \text{in the direction $e$ in $B(0,10(M+1)N)$ that contain} \\
 & \text{a $\psi$-bad $L_0$-box at levels $\gamma > \delta$} \rbrace.
 \end{split}
 \end{equation}
 Note that $\mathcal{B}_N$ is the event from (4.22) of \cite{NS17} with 'bad$(\alpha,\beta,\gamma)$' replaced by $\psi$-bad at levels $\gamma > \delta$, and because of \eqref{eq:SuperExponentialBound} it holds that 
 \begin{equation}
 \label{eq:SuperExpCombined}
 \lim_N \frac{1}{N^{d-2}} \log \pM[\mathcal{B}_N] = -\infty,
 \end{equation}
 similar as in Lemma 4.2 of \cite{NS17}. Thus, $\mathcal{B}_N$ is negligible for our purposes, which allows us to work with the 'effective event' 
 \begin{equation}
 \widetilde{\mathcal{D}}^\alpha_N = \mathcal{D}^\alpha_N \setminus \mathcal{B}_N.
 \end{equation}
 
  We proceed as in \cite{NS17} and extract a maximal (measurable random) subset $\widetilde{\mathcal{S}}_N$ of $\widehat{\mathcal{S}}_N$ with the property that $B(x,2\widehat{L}_0) \cap B(x',2\widehat{L}_0) = \emptyset$ for any $x \neq x'$ in $\widetilde{\mathcal{S}}_N$.  On the complement of $\mathcal{B}_N$, one can for large $N$ choose for every $x \in \widetilde{\mathcal{S}}_N$ a projection $\widetilde{\pi}_x$ on the hyperplane of points with vanishing $\widetilde{i}_x$-coordinate (where $1 \leq  \widetilde{i}_x \leq d$ depends on $\widetilde{\pi}_x$) and a collection $\widetilde{\mathcal{C}}_x$ of $L_0$-boxes intersecting $B(x,\widehat{L}_0)$ with $\widetilde{\pi}_x$-projection at mutual distance $\geq \overline{K}L_0$ and cardinality $\left[ \left( \frac{c'}{K} \frac{\widehat{L}_0}{L_0}\right)^{d-1}\right]$, such that $B_z$ is $\psi$-good at levels $\gamma > \delta$ and $h$-bad at level $a$ for all $z \in \widetilde{\mathcal{C}}_x$, see (4.39)-(4.41) of \cite{NS17} for the details of this construction. \\
 
 As below (4.41) of \cite{NS17}, we thus obtain for large $N$ a random variable
 \begin{equation}
 \label{eq:Kappa_Def}
 \kappa_N = \left(\widehat{\mathcal{S}}_N, \widetilde{\mathcal{S}}_N, (\widetilde{\pi}_x,\widetilde{\mathcal{C}}_x)_{x\in \widetilde{\mathcal{S}}_N} \right) \qquad \text{ on } \widetilde{\mathcal{D}}^\alpha_N (= \mathcal{D}^\alpha_N \setminus \mathcal{B}_N)
 \end{equation}
 ($\kappa_N$ is constructed by selecting from $\widehat{\mathcal{S}}_N$ in a measurable fashion the set $\widetilde{\mathcal{S}}_N$ and then picking $(\widetilde{\pi}_x,\widetilde{\mathcal{C}}_x)_{x \in \mathcal{S}_N}$ according to the procedure explained above, and thus is a well-defined random variable on $\mathcal{D}^\alpha_N \setminus \mathcal{B}_N$).
 Moreover, see (4.43) of \cite{NS17}, the set of possible values of $\kappa_N$, which we call $\mathcal{K}_N$, has cardinality
 \begin{equation}
 \label{eq:Complexity}
 |\mathcal{K}_N| = \exp\left(o(N^{d-2}) \right).
 \end{equation}
 
 We define the coarse graining of the 'good part' $\widetilde{\mathcal{D}}^\alpha_N$ of $\mathcal{D}^\alpha_N$:
 \begin{equation}
 \label{eq:CoarseGraningScheme}
 \widetilde{\mathcal{D}}^\alpha_N = \bigcup_{\kappa \in \mathcal{K}_N} \mathcal{D}_{N, \kappa},\qquad \text{ with } \mathcal{D}_{N,\kappa} = \widetilde{\mathcal{D}} ^\alpha_N \cap \lbrace \kappa_N = \kappa \rbrace, \kappa \in \mathcal{K}_N.
 \end{equation}
 By applying a union bound and using the super-exponential bound \eqref{eq:SuperExpCombined}, we get
 \begin{equation}
 \limsup_N \frac{1}{N^{d-2}} \log \pM[\mathcal{D}^\alpha_N] \leq \limsup_N \sup_{\kappa \in \mathcal{K}_N} \frac{1}{N^{d-2}} \log \pM[\mathcal{D}_{N,\kappa}].
 \end{equation}
 Recall that on the event $\mathcal{D}_{N,\kappa}$ with $\kappa = (\widehat{\mathcal{S}}, \widetilde{\mathcal{S}}, (\widetilde{\pi}_x,\widetilde{\mathcal{C}}_x)_{x \in \widetilde{\mathcal{S}}})$, all $B_z$, $z \in \widetilde{\mathcal{C}}_x$, $x \in \widetilde{\mathcal{S}}$ are $h$-bad at level $a$ and at a mutual distance $\geq \overline{K}L_0$ for large $N$, which implies that
 \begin{equation}
 \pM[\mathcal{D}_{N, \kappa}] \leq \pM\left[ \bigcap_{x \in \widetilde{\mathcal{S}}, z \in \widetilde{\mathcal{C}}_x} \lbrace \inf_{D_z} h_z \leq -a \rbrace \right],
 \end{equation}
 and we can use \eqref{eq:ExponentialControl} to obtain 
 \begin{equation}
 \label{eq:ProofStep1}
 \begin{split}
 \limsup_N \frac{1}{N^{d-2}} & \log \pM[\widetilde{\mathcal{D}}^\alpha_N] \\ 
 & \leq - \liminf_N \frac{1}{N^{d-2}} \inf_{\kappa \in \mathcal{K}_N} \left\lbrace \frac{1}{2}\left(a - \frac{c_4}{K}\sqrt{\frac{|\widetilde{\mathcal{C}}|}{\text{cap}_{\mathbb{Z}^d}(C)}}\right)^2_+ \frac{\text{cap}_{\mathbb{Z}^d}(C)}{\alpha(K)} \right\rbrace,
 \end{split}
 \end{equation}
 where we defined 
 \begin{equation}
 \label{eq:TildeCDef}
 \widetilde{\mathcal{C}} = \bigcup_{x \in \widetilde{\mathcal{S}}} \widetilde{\mathcal{C}}_x,
 \end{equation}
 which is a finite subset of points in $\mathbb{L}_0$ with mutual $|\cdot|_\infty$-distance $\geq \overline{K}L_0$, and $C$ is defined as
 \begin{equation}
 \label{eq:CDef}
 C = \bigcup_{x \in \widetilde{\mathcal{S}}} \bigcup_{z \in \widetilde{\mathcal{C}}_x} B_z \left( = \bigcup_{z \in \widetilde{\mathcal{C}}} B_z \right).
 \end{equation}
 $|\widetilde{\mathcal{C}}|$ can be bounded above as follows: Since $\widetilde{\mathcal{S}} \subseteq \widehat{\mathcal{S}}$ is contained in $B(0,(M+2)N) \cap \widehat{\mathbb{L}}_0$ (this is due to the insulation property \eqref{eq:InsulationProperty}) and the collection of $L_0$-boxes $\widetilde{\mathcal{C}}_x$ intersects $B(x,\widehat{L}_0)$ with $\widetilde{\pi}_x$-projection at mutual distance at least $\geq \overline{K}L_0$ and has cardinality $\left[ \left(\frac{c'}{K}\frac{\widehat{L}_0}{L_0}\right)^{d-1} \right]$, we find that for large $N$
 \begin{equation}
 \label{eq:BoundChat}
 \begin{split}
 |\widetilde{\mathcal{C}}| &\leq c \left(\frac{N}{\widehat{L}_0}\right)^{d} \left(\frac{\widehat{L}_0}{L_0}\right)^{d-1} \stackrel{\eqref{eq:L0Def}}{\leq} c'' \gamma_N^{-\frac{1}{2}} \gamma_N \frac{N^{d-2}}{\log N} = c'' \sqrt{\gamma_N} \frac{N^{d-2}}{\log N}.
 \end{split}
 \end{equation}
 Now, we proceed to associate to $C$ a 'scaled $\mathbb{R}^d$-filling' $\Sigma$, 
 \begin{equation}
 \label{eq:ScaledRdFilling}
 \Sigma = \frac{1}{N} \left(\bigcup_{z \in \widetilde{\mathcal{C}}} (z + [0,L_0]^d ) \right) \subseteq \mathbb{R}^d,
 \end{equation}
 together with $S = \partial U_1$, where $U_1$ is the unbounded component of $\mathbb{R}^d \setminus \Delta_N$ (see \eqref{eq:CompactSubsetDelta}) and one can show that for any $A' \subseteq \mathring{A}$ compact subset and some $l_\ast \geq 0$ (depending on $A$, $A'$), for large $N$ and all $\kappa \in \mathcal{K}_N$, $W_x[H_\Sigma < \tau_{10 \frac{\widehat{L}_0}{N}}] \geq c(K)$ for all $x \in S$. This essentially allows us to use the capacity lower bound \eqref{eq:SolidificationEstimate} to infer that
 \begin{equation}
 \label{eq:SolidificationLowerBound}
 \liminf_N \inf_{\kappa \in \mathcal{K}_N} \text{cap}(\Sigma) \geq \text{cap}(A'). 
 \end{equation}
 We refer to (4.48)-(4.54) of \cite{NS17} for the details of this calculation. To conclude the proof, we use the bound on $|\widetilde{\mathcal{C}}|$ \eqref{eq:BoundChat}, together with the fact that 
\begin{equation}
\label{eq:ComparisonDiscreteBrownianCap}
 \varliminf_K \varliminf_N \inf_{\kappa \in \mathcal{K}_N} \frac{1}{N^{d-2}} \text{cap}_{\mathbb{Z}^d}(C) \geq \frac{1}{d} \varliminf_K \varliminf_N \inf_{\kappa \in \mathcal{K}_N} \text{cap}(\Sigma) \stackrel{\eqref{eq:SolidificationLowerBound}}{\geq} \frac{1}{d}\text{cap}(A') > 0,
\end{equation} 
using Proposition A.1 of \cite{NS17}. So, for large enough $K$ one has
 \begin{equation}
 \begin{split}
 \label{eq:AdditionalTermVanishes}
 0 \leq \limsup_N \sup_{\kappa \in \mathcal{K}_N} \frac{|\widetilde{\mathcal{C}}|}{\text{cap}_{\mathbb{Z}^d}(C)} & \leq \limsup_N \sup_{\kappa \in \mathcal{K}_N} \frac{c''\sqrt{\gamma_N}N^{d-2}/ \log N}{\text{cap}_{\mathbb{Z}^d}(C)} \\
 & \leq \frac{\lim_N c''\sqrt{\gamma_N} / \log N}{\liminf_N \inf_{ \kappa \in \mathcal{K}_N} \frac{\text{cap}_{\mathbb{Z}^d}(C)}{N^{d-2}}} = 0,
 \end{split}
\end{equation}  
that is $\lim_N \sup_{\kappa \in \mathcal{K}_N} \frac{|\widetilde{\mathcal{C}}|}{\text{cap}_{\mathbb{Z}^d}(C)} = 0$. From \eqref{eq:ProofStep1}, we obtain for large $K$ that
\begin{equation}
\limsup_N \frac{1}{N^{d-2}} \log \pM[\mathcal{D}^\alpha_N] \leq - \liminf_N \frac{1}{N^{d-2}} \inf_{\kappa \in \mathcal{K}_N} \frac{1}{2} \frac{a^2}{\alpha(K)} \text{cap}_{\mathbb{Z}^d}(C).
\end{equation}
Taking $\limsup_K$ and using \eqref{eq:ComparisonDiscreteBrownianCap} (recall that $\lim_{K \rightarrow \infty} \alpha(K) = 1$), we arrive at
\begin{equation}
\limsup_N \frac{1}{N^{d-2}} \log \pM[\mathcal{D}^\alpha_N] \leq - \frac{1}{2d} a^2 \text{cap}(A').
\end{equation}
Letting $a$ tend to $\overline{h}-\alpha$ and taking $A' \uparrow \mathring{A}$ now concludes the proof.
 \end{proof}
 \begin{assumption}
 1) As already stated in the introduction, if $h_{\ast\ast} = \overline{h}$ and if $A$ is regular in the sense that $\text{cap}(A) = \text{cap}(\mathring{A})$, the combination of Theorems \ref{Theorem21} and \ref{Theorem31} would give 
 \begin{equation}
 \lim_N \frac{1}{N^{d-2}} \log \pM[\mathcal{D}^\alpha_N] = - \frac{1}{2d}(h_\ast-\alpha)^2\text{cap}(A).
 \end{equation}
 2) For each choice of rationals $\delta < \gamma < \overline{h}$ and integer $K \geq 100$, choose $\gamma_N$ as above \eqref{eq:L0Def}.
 The proof of Theorem \ref{Theorem31} could have been done with $A_N$ replaced by $\widehat{A}_N = (NA(\epsilon_N)) \cap \mathbb{Z}^d$ with some sequence $\epsilon_N \rightarrow 0$ such that $\epsilon_N / (\widehat{L}_0 / N)\rightarrow \infty$ (i.e. $\epsilon_N/\sqrt{\gamma_N} \rightarrow \infty$) as $N \rightarrow \infty$, for all above choices of $\gamma_N$ (this can be done by a diagonalization procedure), where $A(\epsilon_N)$ denotes the closed $\epsilon_N$-neighborhood of $A$ in $|\cdot|_\infty$-norm, see Remark 4.5, 3) of \cite{NS17} for details. 
 That way, one can obtain the result
 \begin{equation}
  \limsup_N \frac{1}{N^{d-2}} \log \pM[ \widehat{A}_N\stackrel{\geq \alpha}{\nleftrightarrow} S_N ] \leq - \frac{1}{2d}(\overline{h} - \alpha)^2 \text{cap}(A).
 \end{equation}
 \hfill $\square$
 \end{assumption} 
 \section{Entropic repulsion by disconnection}
 In this final section we derive an asymptotic upper bound on the probability of the event that the level set below $\alpha$ disconnects $A_N$ from $S_N$ \textit{and} that the average of the Gaussian free field over the discrete blow-up of a non-empty open subset $V$ of $\mathring{A}$ with $\overline{V} \subseteq \mathring{A}$ becomes bigger than $-(\overline{h} - \alpha) + \Delta$, for some $\Delta > 0$. Throughout this section, we will consider a compact subset $A$ that fulfills \eqref{eq:ContainedInBall} and assume that $\mathring{A}$ is non-empty. \\ 
  
Let us first give an outline of the proof of the main result of this section, Theorem \ref{Theorem43}. Crucially, we will make use of a modified version of Corollary 4.4 of \cite{S15}, see Lemma \ref{CrucialLemma} below. Recall that given $\alpha < \overline{h}$, we choose parameters $a$, $\delta$ and $\gamma$ according to \eqref{eq:ConventionParameters}. By coarse-graining of the 'good part' $\widetilde{\mathcal{D}}^\alpha_N$ of the disconnection event, one is brought to consider a $\kappa \in \mathcal{K}_N$ (see \eqref{eq:Kappa_Def} and \eqref{eq:CoarseGraningScheme}) that gives rise to a union $C$ of boxes, located at points in $\widetilde{\mathcal{C}}$ at mutual distance $\geq \overline{K}L_0$, which are all $h$-bad at level $a$ (see \eqref{eq:TildeCDef} and \eqref{eq:CDef}). Corollary 4.4 of \cite{S15} (see \eqref{eq:ExponentialControl}) gives an asymptotic exponential upper bound on the probability of the event that in a large finite collection of such well-separated boxes, all are $h$-bad at some level $a > 0$. Given $\Delta > 0$, we show that the event that this happens \textit{and} the average of the Gaussian free field over the discrete blow-up of $V$ is above $-a+ \Delta$ for $a$ as in \eqref{eq:ConventionParameters} is contained in $\lbrace \inf_{f \in \mathcal{F}} \widehat{Z}_f \leq - a - \beta \Delta \rbrace$ for some $\beta > 0$, where $(\widehat{Z}_f)_{f \in \mathcal{F}}$ is a certain Gaussian process (in fact, a slightly modified version of $(Z_f)_{f \in \mathcal{F}}$ from Section 4 of \cite{S15}). By bounding from above the variance of $\widehat{Z}_f$, we obtain a modification of \eqref{eq:ExponentialControl} using the Borell-TIS inequality, see \eqref{eq:ModifiedExponentialControl}. \\
\\
Our first step is a proposition that shows that with probability tending to $1$ as $N \rightarrow \infty$, the simple random walk started at a point well inside $A_N$ must hit the set $C$  of surrounding boxes for any choice of $\kappa \in \mathcal{K}_N$. The proof uses the solidification estimate for Brownian motion and a strong coupling result of \cite{E89} in the spirit of Komlós, Major and Tusnády that couples the simple random walk and Brownian motion. \\

\begin{proposition}
\label{Prop41}
Consider a compact subset $A' \subseteq \mathring{A}$ and its discrete blow-up $A'_N = (NA') \cap \mathbb{Z}^d$. Then one has:
\begin{equation}
\liminf_{N \rightarrow \infty} \inf_{\kappa \in \mathcal{K}_N} \inf_{x\in A'_N} P_x[H_C < \infty] \geq 1.
\end{equation}
\begin{proof}
For $\kappa \in \mathcal{K}_N$, we introduce the sets 
\begin{equation}
\widetilde{\Gamma} = \bigcup_{z \in \widetilde{\mathcal{C}}} (z + [3L_0/8,5L_0/8]^d), \qquad \widetilde{\Sigma} = \frac{1}{N} \widetilde{\Gamma} (\subseteq \Sigma),
\end{equation}
$\Sigma$ being the scaled $\mathbb{R}^d$-filling of $C$ from $\eqref{eq:ScaledRdFilling}$. We note that $\widetilde{\Sigma}$ fulfills the solidification estimate \eqref{eq:SolidificationProb}, since one can show analogously to $\Sigma$ that for large $N$ and all $\kappa \in \mathcal{K}_N$, $W_x[H_{\widetilde{\Sigma}} < \tau_{10 \frac{\widehat{L}_0}{N}}] \geq c(K)$ for all $x \in \partial U_1$, where $U_1$ is defined as below \eqref{eq:ScaledRdFilling}. Together with the scaling invariance of Brownian motion, this will allow us to show that $W_x[H_{\widetilde{\Gamma}} < \infty]$ is close to $1$ uniformly in $x \in A_N'$ and $\kappa \in \mathcal{K}_N$, as $N \rightarrow \infty$. The claim will follow by bounding this probability from above by the sum of $\inf_{\kappa \in \mathcal{K}_N} \inf_{x \in A_N'}P_x[H_C < \infty]$ and some error that converges to $0$ as $N \rightarrow \infty$, see \eqref{eq:BoundW} and the argument following it. \\  

Let $\overline{B}(x)$ denote the closed Euclidean ball in $\mathbb{R}^d$ with radius $10d(M+2)$, centered at $x$. For a closed set $F \subseteq \mathbb{R}^d$, we denote by $L_{F} = \sup \lbrace 0 < t < \infty; Z_t \in F\rbrace$ the time of last visit of the Brownian motion to $F$ (with the convention that $L_F = 0$ if the set on the right-hand side is empty). Clearly, $\widetilde{\Gamma} \subseteq N\overline{B}(x)$ for every $x \in A_N'$ by the first part of \eqref{eq:InsulationProperty}, so $W_x$-a.s. $H_{\widetilde{\Gamma}} \leq L_{N\overline{B}(x)}$ on $\lbrace H_{\widetilde{\Gamma}}<\infty \rbrace$. 
 We fix some $\varepsilon > 0$ and obtain by scaling and translation invariance as well as transience of Brownian motion that $W_x[L_{N\overline{B}(x)} > N^{2 + \varepsilon}] = W_0[L_{\overline{B}(0)} > N^\varepsilon] \rightarrow 0$ and thus:
\begin{equation}
\label{eq:Bound1}
W_x[H_{\widetilde{\Gamma}} < \infty] \leq W_x[H_{\widetilde{\Gamma}} \leq N^{2+\varepsilon}] + W_0[L_{\overline{B}(0)} > N^{\varepsilon}],
\end{equation}
and the second term on the right-hand side converges to $0$ as $N \rightarrow \infty$. Since the coupling result we apply below allows us to compare the simple random walk at integer times and Brownian motion at times that are integer multiples of $\frac{1}{d}$, we define $\widehat{H}_{\widetilde{\Gamma}}$ as the smallest integer multiple of $\frac{1}{d}$ bigger or equal to $H_{\widetilde{\Gamma}}$ and control the $|\cdot|_\infty$-distance of the Brownian motion at these two random times. 
On the event $\lbrace H_{\widetilde{\Gamma}} \leq N^{2+\varepsilon}\rbrace$, we apply the strong Markov property at $H_{\widetilde{\Gamma}}$, to obtain
\begin{equation}
\label{eq:Bound2}
W_x[H_{\widetilde{\Gamma}} \leq  N^{2+\varepsilon}] \leq  W_x[H_{\widetilde{\Gamma}} \leq N^{2+\varepsilon}, |Z_{H_{\widetilde{\Gamma}}} - Z_{\widehat{H}_{\widetilde{\Gamma}}}|_\infty < \tfrac{L_0}{8}] +W_0\bigg[\sup_{0 \leq t \leq 1/d} |Z_t|_\infty\geq \tfrac{L_0}{8}\bigg],
\end{equation}
and note that the last term tends to $0$ as $N \rightarrow \infty$.
 By Theorem 4 of \cite{E89}, there exists a probability space $(\widetilde{\Omega}, \widetilde{\mathcal{F}},\widetilde{P})$, a simple random walk $(\widetilde{X}_n)_{n \geq 0}$ on $\mathbb{Z}^d$ and a Brownian motion $(\widetilde{Z}_t)_{t \geq 0}$ in $\mathbb{R}^d$, both started in $x$, on this probability space such that:
 \begin{equation}
 \label{eq:Bound3}
 \widetilde{P}\left[ \max_{1 \leq k \leq 2d[N^{2+\varepsilon}]} |\widetilde{X}_k - \widetilde{Z}_{k/d} |_\infty \geq \tfrac{L_0}{8} \right] \rightarrow 0, \text{ as } N \rightarrow \infty.
 \end{equation}
 We denote by $H^{\widetilde{Z}}_{\widetilde{\Gamma}}$ and $H^{\widetilde{X}}_{C}$ the entrance times into $\widetilde{\Gamma}$ and $C$, associated to $\widetilde{Z}_\cdot$ and $\widetilde{X}_\cdot$, respectively. By collecting \eqref{eq:Bound1}-\eqref{eq:Bound3}, we obtain (with hopefully obvious notation):
 \begin{equation}
 \begin{split}
 \label{eq:BoundW}
 W_x[H_{\widetilde{\Gamma}} < \infty] \leq \widetilde{P}& \left[H^{\widetilde{Z}}_{\widetilde{\Gamma}} \leq N^{2+\varepsilon}, \max_{1 \leq k \leq 2d[N^{2+\varepsilon}]} |\widetilde{X}_k - \widetilde{Z}_{k/d} |_\infty < \tfrac{L_0}{8}, |\widetilde{Z}_{H^{\widetilde{Z}}_{\widetilde{\Gamma}}} - \widetilde{Z}_{\widehat{H}^{\widetilde{Z}}_{\widetilde{\Gamma}}}|_\infty < \tfrac{L_0}{8} \right] \\ 
 & + o(1) \qquad \text{ as $N \rightarrow \infty$, uniformly in $\kappa \in \mathcal{K}_N$, $x \in A_N$.}
 \end{split}
 \end{equation}
 For large $N$, the event under the probability on the right-hand side of \eqref{eq:BoundW} is contained in $\lbrace H^{\widetilde{X}}_C < \infty\rbrace$, and by scaling invariance of Brownian motion, the left-hand side can be bounded uniformly from below by $\inf_{\kappa \in \mathcal{K}_N} \inf_{x \in A'} W_x[H_{\widetilde{\Sigma}} < \infty]$. The result now follows by taking in \eqref{eq:BoundW} the infimum over $x \in A_N'$ and over $\kappa \in \mathcal{K}_N$, letting $N \rightarrow \infty$ and using the solidification estimate for porous interfaces $\widetilde{\Sigma}$ from \eqref{eq:SolidificationProb} (see also (4.49)-(4.51) of \cite{NS17}).
\end{proof}
\end{proposition}
The next lemma will play a crucial role in the proof of the main result of this section, Theorem \ref{Theorem43}. It gives a bound on the probability that $\inf_{D_z} h_z \leq -a$ for all $z \in \widetilde{\mathcal{C}}$ (recall the definitions of $D_z$ and $h_z$ from \eqref{eq:Boxes} and below \eqref{eq:Boxes}) and that the average of the Gaussian free field over the discrete blow-up of an open set well inside $A$ stays above $-a + \Delta$, with $a$ from \eqref{eq:ConventionParameters} and $\Delta > 0$. 
\begin{lemma}
\label{CrucialLemma}
Consider $\Delta > 0$, $V$ a non-empty open subset of $\mathring{A}$ with $\overline{V} \subseteq \mathring{A}$ and $\beta > 0$. There exists a function $\widetilde{\alpha}(\cdot)$ (dependent on $\beta$) with $\lim_K \widetilde{\alpha}(K) = 1$, see \eqref{eq:AlphaTildeDef}, such that for large enough $K$: 
 \begin{equation}
 \label{eq:ModifiedExponentialControl}
 \begin{split}
 \limsup_N \sup_{\kappa \in \mathcal{K}_N} & \bigg\lbrace \log \pM\bigg[ \bigcap_{z \in \widetilde{\mathcal{C}}} \lbrace \inf_{D_z} h_z \leq -a \rbrace \cap \bigg\lbrace \frac{1}{|V_N|} \sum_{x \in V_N} \varphi_x \geq -a + \Delta \bigg\rbrace \bigg]  \\
 & + \frac{1}{2}\left( a + \beta \Delta - \frac{c_5}{K}\sqrt{\frac{|\widetilde{\mathcal{C}}|}{\text{\normalfont{cap}}_{\mathbb{Z}^d}(C)}}\right)^2_+ \frac{\text{\normalfont{cap}}_{\mathbb{Z}^d}(C)}{\widetilde{\alpha}(K)+ \beta^2 c_6(V,M)} \bigg\rbrace \leq 0.
 \end{split}
\end{equation}
\end{lemma}
\begin{proof}
For $\kappa \in \mathcal{K}_N$, set $\mathcal{F} = \lbrace f \in (\mathbb{Z}^d)^{\widetilde{\mathcal{C}}}; f(z) \in D_z \text{ for each } z \in \widetilde{\mathcal{C}} \rbrace$ and define
\begin{equation}
\begin{split}
Z_f & = \sum_{z \in \widetilde{\mathcal{C}}} \lambda(z) h_z(f(z)), \\
\widehat{Z}_f &= Z_f - \beta \bigg( \tfrac{1}{|V_N|}\sum_{x \in V_N} \varphi_x - Z_f \bigg) = (1+\beta)\sum_{z \in \widetilde{\mathcal{C}}} \lambda(z) h_z(f(z)) - \beta  \tfrac{1}{|V_N|}\sum_{x \in V_N} \varphi_x ,
\end{split}
\end{equation}
where 
\begin{equation}
\label{eq:DefLambda}
\lambda(z) = \overline{e}_C(B_z) = \frac{e_C(B_z)}{\text{cap}_{\mathbb{Z}^d}(C)}.
\end{equation}
Note that $\widehat{Z}_f$ is a zero-average Gaussian field and 
\begin{equation}
\text{var}(\widehat{Z}_f) = (1+\beta)^2 \text{var}(Z_f) - 2 \beta(1+\beta)\mathcal{G}_N + \beta^2 \mathcal{H}_N,
\end{equation}
where we defined 
\begin{equation}
\mathcal{G}_N = \mathbb{E}\bigg[Z_f \tfrac{1}{|V_N|} \sum_{x \in V_N} \varphi_x \bigg], \text{ and }
\mathcal{H}_N = \mathbb{E}\bigg[\bigg( \tfrac{1}{|V_N|}\sum_{x \in V_N} \varphi_x \bigg)^2 \bigg].
\end{equation}
From the proof of Theorem 4.2 of \cite{S15} (see (4.25) of this reference), we have
\begin{equation}
\label{eq:BoundZf1}
\text{var}(Z_f) \leq \frac{1}{\text{cap}_{\mathbb{Z}^d}(C)} \left( \frac{c}{K^{d-2}} + \gamma(K,L_0) \right),
\end{equation}
where 
\begin{equation}
\gamma(K,L_0) = \sup_{\substack{z,z' \in \mathbb{L}_0 \\ |z-z'|_\infty \geq (2K+1)L_0}} \sup_{\substack{y \in D_z, y' \in D_{z'} \\ x \in B_z, x' \in B_{z'} }}  \frac{g(y,y')}{g(x,x')} (\geq 1),
\end{equation}
and $\lim_K \limsup_N \gamma(K,L_0) = 1$ (see (4.21), (4.22) of \cite{S15}). For later use, we define 
\begin{align}
\label{eq:DefGammaTilde}
\widetilde{\gamma}(K,L_0) & = \inf_{\substack{z,z' \in \mathbb{L}_0 \\ |z-z'|_\infty \geq (2K+1)L_0}} \inf_{\substack{y \in D_z, y' \in D_{z'} \\ x \in B_z, x' \in B_{z'} }}  \frac{g(y,y')}{g(x,x')} (\leq 1), \\
\alpha_1(K) &  = \frac{c}{K^{d-2}} +  \limsup_N  \gamma(K,L_0), \\
\label{eq:Alpha2}
\alpha_2(K) & = \liminf_N \widetilde{\gamma}(K,L_0),
\end{align} 
and note that $\lim_K \alpha_1(K) = \lim_K \alpha_2(K) = 1$. 
To conclude the proof of the lemma, we have to investigate the asymptotics of $\mathcal{G}_N$ and $\mathcal{H}_N$ as $N \rightarrow \infty$ and if $K$ is large (notice that $\mathcal{G}_N$ depends implicitly on $K$ through $Z_f$). \\
\\
We begin with $\mathcal{G}_N$. For large $N$ (since $N/L_0 \rightarrow \infty$), every point in $V_N$ has a distance $\geq (2K+1)L_0$ from every box in $\widetilde{\mathcal{C}}$, so we obtain 
\begin{equation}
\begin{split}
\label{eq:GNdifference}
\mathcal{G}_N & = \frac{1}{|V_N|}\sum_{x \in V_N}\sum_{z \in \widetilde{\mathcal{C}}} g(x,f(z)) \lambda(z) \\
& \stackrel{\eqref{eq:DefLambda},\eqref{eq:DefGammaTilde}}{\geq}  \frac{\widetilde{\gamma}(K,L_0)}{|V_N|\text{cap}_{\mathbb{Z}^d}(C)} \sum_{x \in V_N} \sum_{x' \in C} g(x,x')e_C(x') \stackrel{\eqref{eq:EquilibriumPotential}}{=} \frac{\widetilde{\gamma}(K,L_0)}{|V_N|\text{cap}_{\mathbb{Z}^d}(C)} \sum_{x \in V_N} P_x[H_C < \infty].
\end{split}
\end{equation} 
Using Proposition \ref{Prop41}, we conclude from \eqref{eq:GNdifference} that
\begin{equation}
\label{eq:BoundZf2}
\liminf_{N } \inf_{\kappa \in \mathcal{K}_N} (\text{cap}_{\mathbb{Z}^d}(C) \mathcal{G}_N) \geq \liminf_N \widetilde{\gamma}(K,L_0) \stackrel{\eqref{eq:Alpha2}}{=} \alpha_2(K).
\end{equation}
We now turn to $\mathcal{H}_N$: We have
\begin{equation}
\begin{split}
\mathcal{H}_N & = \mathbb{E}\bigg[\bigg( \tfrac{1}{|V_N|} \sum_{x \in V_N} \varphi_x \bigg)^2 \bigg]  = \tfrac{1}{|V_N|^2} \sum_{x,y \in V_N} g(x,y) \\
& = \frac{1}{N^{2d}(1 + r_N)} \left\langle \Phi_N, g\Phi_N \right\rangle_{A_N},
\end{split}
\end{equation}
with some $r_N \rightarrow 0$, where we defined $\Phi = \frac{1}{|V|} \mathbbm{1}_V$, $\Phi_N = \Phi\left(\frac{\cdot}{N}\right)$ and 
\begin{equation}
\left\langle a, gb \right\rangle_{A_N} = \sum_{x,y \in A_N} a(x)g(x,y)b(y),
\end{equation}
for functions $a,b : \mathbb{Z}^d \rightarrow \mathbb{R}$ supported in $A_N$. The function $\Phi$ is Riemann-integrable, and it follows from Lemma 2.2 of \cite{BD93} (stated for continuous functions $\Phi$ in this reference, but its proof remains valid if $\Phi$ is Riemann-integrable) that 
\begin{equation}
\lim_{N \rightarrow \infty} \frac{1}{N^{d+2}} \left\langle \Phi_N, g\Phi_N \right\rangle_{A_N} = c \int \int_{A \times A} \frac{\Phi(x) \Phi(y)}{|x-y|^{d-2}} dx dy = W(V) \in (0,\infty), 
\end{equation}
which can be understood as the electrostatic energy of a unit charge uniformly distributed on $V$. 
We thus obtain the following asymptotic upper bound: 
\begin{equation}
\label{eq:BoundZf3}
\begin{split}
\limsup_N \sup_{\kappa \in \mathcal{K}_N} (\text{cap}_{\mathbb{Z}^d}(C) \mathcal{H}_N) & \leq \limsup_{N}\sup_{\kappa \in \mathcal{K}_N}\frac{\text{cap}_{\mathbb{Z}^d}(C)}{N^{d-2}(1 + r_N)} \frac{1}{N^{d+2}} \left\langle \Phi_N, g\Phi_N \right\rangle_{A_N} \\
& \leq c'M^{d-2} W(V) = c_6(V,M),
\end{split}
\end{equation}
where we used in the last inequality the monotonicty of capacity, $C \subseteq B(0,10(M+2)N)$ for large $N$ and \eqref{eq:BoxCapBound}. We set
\begin{equation}
\label{eq:AlphaTildeDef}
\widetilde{\alpha}(K) = \alpha_1(K) + 2\beta(\alpha_1(K) - \alpha_2(K)) \ (\rightarrow 1 \text{ as $K \rightarrow \infty$, see below \eqref{eq:Alpha2}}).
\end{equation}
Collecting the bounds in \eqref{eq:BoundZf1}, \eqref{eq:BoundZf2} and \eqref{eq:BoundZf3} we finally arrive at the following upper bound for the variance of $\widehat{Z}_f$: 
\begin{equation}
\begin{split}
\limsup_{N \rightarrow \infty}\sup_{\kappa \in \mathcal{K}_N} (\text{cap}_{\mathbb{Z}^d}(C)\text{var}(\widehat{Z}_f)) &\leq \alpha_1(K)(1+\beta)^2 - 2\beta(1+\beta)\alpha_2(K) + \beta^2 c_6(V,M) \\
& \leq \widetilde{\alpha}(K) + \beta^2c_6(V,M)
\end{split}
\end{equation}
for large enough $K$. We conclude as in the proof of Corollary 4.4 of \cite{S15}: To this end, define
\begin{equation}
\widehat{Z} = \inf_{f \in \mathcal{F}} \widehat{Z}_f,
\end{equation}
and note that $\mathbb{E}[\widehat{Z}] = \mathbb{E}[\inf_{f \in \mathcal{F}} Z_f]$, and again from Theorem 4.2 of \cite{S15}, one has an upper bound on this quantity, which reads:
\begin{equation}
\sup_{\kappa \in \mathcal{K}_N} |\mathbb{E}[\inf_{f \in \mathcal{F}} Z_f]| \left( \frac{|\widetilde{\mathcal{C}|}}{\text{cap}_{\mathbb{Z}^d}(C)}\right)^{-\frac{1}{2}} \leq \frac{c_5}{K}.
\end{equation}
Note that the event under the probability in \eqref{eq:ModifiedExponentialControl} is contained in $\lbrace \inf_{f \in \mathcal{F}} \widehat{Z}_f \leq -a - \beta \Delta \rbrace$. 
Using the Borell-TIS inequality (see Theorem 2.1.1, p. 50 of \cite{AT}), one obtains
\begin{equation}
\begin{split}
& \pM\bigg[ \bigcap_{z \in \widetilde{\mathcal{C}}} \lbrace \inf_{D_z} h_z \leq -a \rbrace \cap \bigg\lbrace \frac{1}{|V_N|} \sum_{x \in V_N} \varphi_x \geq -a + \Delta \bigg\rbrace  \bigg] \\
& \qquad \qquad \qquad\leq \pM\left[\inf_{f\in \mathcal{F}}\widehat{Z}_f \leq -a - \beta\Delta \right]  \leq \exp\left\lbrace - \frac{1}{2\sigma^2}\left(a + \beta \Delta - |\mathbb{E}[\widehat{Z}]| \right)_+^2 \right\rbrace,
\end{split}
\end{equation}
where $\sigma^2 = \sup_\mathcal{F} \text{var}(\widehat{Z}_f)$. The claim now follows as in the proof of Corollary 4.4 of \cite{S15} by taking logarithms and inserting the bounds on $\text{var}(\widehat{Z}_f)$ and $|\mathbb{E}[\widehat{Z}]|$. 
\end{proof}
We now have all the elements to state and prove the main result alluded to in the beginning of this section.
\begin{theorem}
\label{Theorem43}
Consider $\Delta > 0$, $\alpha < \overline{h}$ and a non-empty open subset $V$ of $\mathring{A}$ with $\overline{V} \subseteq \mathring{A}$. Then, for any $\beta > 0$ one has 
\begin{equation}
\label{eq:Theorem42claim}
\begin{split}
\limsup_N & \frac{1}{N^{d-2}}\log\pM\bigg[ \bigg\lbrace \frac{1}{|V_N|} \sum_{x \in V_N} \varphi_x \geq -(\overline{h}-\alpha) + \Delta \bigg\rbrace \cap \mathcal{D}^\alpha_N \bigg] \\ 
& \qquad\leq - \frac{1}{2d}\left( \overline{h} - \alpha + \beta \Delta\right)^2 \frac{\text{\normalfont{cap}}(\mathring{A})}{1+ \beta^2 c_6(V,M)} 
\end{split}
\end{equation}
\end{theorem}
\begin{proof}
We denote by $\mathcal{J}_N$ the event under the probability in \eqref{eq:Theorem42claim}. Choose $\alpha + a = \delta < \gamma < \overline{h}$ so that $\widetilde{\Delta} = \Delta - (\overline{h} - \delta) > 0$. 
By using the coarse-graining of the event $\widetilde{\mathcal{D}}^\alpha_N$ as in \eqref{eq:CoarseGraningScheme}, one obtains by Lemma \ref{CrucialLemma} that
\begin{equation}
\begin{split}
\varlimsup_N \frac{1}{N^{d-2}} \log \pM[\mathcal{J}_N] & \leq \varlimsup_N \frac{1}{N^{d-2}} \sup_{\kappa \in \mathcal{K}_N} \log \pM\bigg[\bigcap_{x \in \widetilde{\mathcal{C}}} \lbrace \inf_{D_z} h_z \leq -a \rbrace \cap \bigg\lbrace \frac{1}{|V_N|} \sum_{x \in V_N} \varphi_x \geq -a + \widetilde{\Delta} \bigg\rbrace \bigg] \\& \leq - \varliminf_N \frac{1}{N^{d-2}} \inf_{\kappa \in \mathcal{K}_N} \frac{1}{2}\left( a + \beta \widetilde{\Delta} - \frac{c_5}{K}\sqrt{\frac{|\widetilde{\mathcal{C}}|}{\text{\normalfont{cap}}_{\mathbb{Z}^d}(C)}}\right)^2_+ \frac{\text{\normalfont{cap}}_{\mathbb{Z}^d}(C)}{\widetilde{\alpha}(K)+ \beta^2 c_6(V,M)}. 
\end{split}
\end{equation}
Now, one can follow the same steps as in the proof of Theorem \ref{Theorem31} after \eqref{eq:ProofStep1} to obtain the right-hand side of \eqref{eq:Theorem42claim}: Indeed, using \eqref{eq:AdditionalTermVanishes}, we find
\begin{equation}
\varlimsup_N \frac{1}{N^{d-2}} \log \pM[\mathcal{J}_N] \leq  - \varliminf_N \frac{1}{N^{d-2}} \inf_{\kappa \in \mathcal{K}_N} \frac{1}{2}\left( a + \beta \widetilde{\Delta}\right)^2 \frac{\text{\normalfont{cap}}_{\mathbb{Z}^d}(C)}{\widetilde{\alpha}(K)+ \beta^2 c_6(V,M)}
\end{equation}
and by taking $\limsup_K$ on both sides and using \eqref{eq:ComparisonDiscreteBrownianCap} (recall that $\lim_{K \rightarrow \infty}\widetilde{\alpha}(K) = 1$), we find that for any compact subset $A' \subseteq \mathring{A}$: 
\begin{equation}
\limsup_N \frac{1}{N^{d-2}} \log \pM[\mathcal{J}_N] \leq - \frac{1}{2d}\left(a + \beta \widetilde{\Delta}\right)^2 \frac{\text{cap}(A')}{1 + \beta^2 c_6(V,M)}.
\end{equation}
The claim now follows by letting $a$ tend to $\overline{h} - \alpha$ and taking $A' \uparrow \mathring{A}$.
\end{proof}
We conclude this section with a corollary that combines the asymptotic upper bound from Theorem \ref{Theorem43} with the asymptotic lower bound from Theorem \ref{Theorem21} under the assumption that the critical parameters $\overline{h},h_{\ast}$ and $h_{\ast\ast}$ coincide.
\begin{corollary}
\label{LastCor}
Consider $\Delta$, $\alpha$ and $V$ as in Theorem \ref{Theorem43} and assume that $A$ is regular in the sense that $\text{\normalfont{cap}}(A) = \text{\normalfont{cap}}(\mathring{A})$. Then, under the assuption that $\overline{h} = h_\ast = h_{\ast\ast}$, one has
\begin{equation}
\lim_{N \rightarrow \infty} \pM\left[  \frac{1}{|V_N|} \sum_{x \in V_N} \varphi_x \geq -(h_\ast-\alpha) + \Delta  \bigg\vert \mathcal{D}^\alpha_N \right] = 0.
\end{equation}
\end{corollary}
\begin{proof}
For every $\Delta > 0$, one can find $\beta > 0$ small such that
\begin{equation}
\frac{(h_\ast - \alpha + \beta\Delta)^2}{1 + \beta^2 c_6(V,M)} > (h_\ast - \alpha)^2.
\end{equation}
With the notation from the proof of Theorem \ref{Theorem43}, we obtain: 
\begin{equation}
\limsup_N \frac{1}{N^{d-2}} \log \frac{\pM[\mathcal{J}_N]}{\pM[\mathcal{D}^\alpha_N]} \leq \limsup_N \frac{1}{N^{d-2}} \log \pM[\mathcal{J}_N] - \liminf_N \frac{1}{N^{d-2}} \log \pM[\mathcal{D}^\alpha_N] < 0,
\end{equation}
using in the last step Theorem \ref{Theorem43}, Theorem \ref{Theorem21}, the regularity condition on $A$ and the assumption that $\overline{h} = h_\ast = h_{\ast\ast}$.
\end{proof}

\textbf{ } \\
 \textbf{Acknowledgement. }The author wishes to thank A.-S. Sznitman for suggesting the problem, for helpful discussions and comments on earlier drafts of this article, as well as A. Chiarini for useful discussions.
 \selectlanguage{english}

 \end{document}